\documentclass[envcountsect, envcountsame]{svmult}

\usepackage{mathptmx}       
\usepackage{helvet}         
\usepackage{courier}        
\usepackage{type1cm}        
\usepackage{makeidx}         
\usepackage{graphicx}        
\usepackage{multicol}        
\usepackage[bottom]{footmisc}

\makeindex             

\usepackage{amstext,amsfonts,amssymb,amscd,amsbsy,amsmath}
\usepackage{tikz-cd}	
\usepackage{tikz}

\smartqed

\newcommand{\A}{\mathbb{A}}

\newcommand{\NN}{\mathbb{N}}
\newcommand{\PP}{\mathbb{P}}

\newcommand{\ZZ}{\mathbb{Z}}
\newcommand{\CC}{\mathbb{C}}

\DeclareMathOperator{\GL}{GL}

\DeclareMathOperator{\Pic}{Pic}
\DeclareMathOperator{\transpose}{^T}
\DeclareMathOperator{\variety}{V}

\DeclareMathOperator{\Gr}{Gr}

\begin{document}
\title*{Secants, Bitangents, and Their Congruences}
\author{Kathl\'en Kohn, Bernt Ivar Utst{\o}l N{\o}dland, and Paolo Tripoli}
\institute{
  Kathl\'en Kohn \at 
  Institute of Mathematics, Technische Universit\"at Berlin, Sekretariat MA
  6-2, Stra{\ss}e des 17. Juni 136, 10623 Berlin, Germany,
  \email{kohn@math.tu-berlin.de}
  \and 
  Bernt Ivar Utst\o{}l N\o{}dland \at 
  Department of Mathematics, University of Oslo, Moltke Moes vei 35, Niels
  Henrik Abels hus, 0851 Oslo, Norway, \email{berntin@math.uio.no}
  \and 
  Paolo Tripoli \at 
  Mathematics Institute, University of Warwick, Zeeman Building, Coventry, CV4
  7AL, United Kingdom, \email{paolo.tripoli@nottingham.ac.uk}}

\maketitle

\abstract{ A congruence is a surface in the Grassmannian\index{Grassmannian}
  $\Gr(1,\PP^3)$ of lines in projective $3$-space. To a space curve $C$, we
  associate the Chow hypersurface in $\Gr(1,\PP^3)$ consisting of all lines
  which intersect $C$. We compute the singular locus of this hypersurface,
  which contains the congruence of all secants to $C$. A surface $S$ in
  $\PP^3$ defines the Hurwitz hypersurface in $\Gr(1,\PP^3)$ of all lines
  which are tangent to $S$. We show that its singular locus has two components
  for general enough $S$: the congruence of bitangents and the congruence of
  inflectional tangents. We give new proofs for the bidegrees of the secant,
  bitangent and inflectional congruences, using geometric techniques such as
  duality, polar loci and projections.  We also study the singularities of
  these congruences.}

\section{Introduction}
\label{05:start}
\label{05sec:intro}

The aim of this article is to study subvarieties of Grassmannians which arise
naturally from subvarieties of complex projective $3$-space $\PP^3$.  We are
mostly interested in threefolds and surfaces in $\Gr(1,\PP^3)$.  These are
classically known as \emph{line complexes} and
\emph{congruences}\index{congruence}.  We determine their classes in the Chow
ring\index{Chow!ring} of $\Gr(1,\PP^3)$ and their singular loci.  Throughout
the paper, we use the phrase `singular points of a congruence' to simply refer
to its singularities as a subvariety of the Grassmannian $\Gr(1,\PP^3)$. In
older literature, this phrase refers to points in $\PP^3$ lying on infinitely
many lines of the congruence; nowadays, these are called \emph{fundamental
  points}.

The \emph{Chow hypersurface}\index{Chow!hypersurface}
$\operatorname{CH}_0(C) \subset \Gr(1,\PP^3)$ of a curve $C \subset \PP^3$ is
the set of all lines in $\PP^3$ that intersect $C$, and the \emph{Hurwitz
  hypersurface}\index{Hurwitz~hypersurface}
$\operatorname{CH}_1(S) \subset \Gr(1,\PP^3)$ of a surface $S \subset \PP^3$
is the Zariski closure of the set of all lines in $\PP^3$ that are tangent to
$S$ at a smooth point.  Our main results are consolidated in the following
theorem.

\begin{theorem} 
  \label{05mainthm} 
  Let $C \subset \PP^3$ be a nondegenerate curve\index{curve!space} of degree
  $d$ and geometric genus $g$ having only ordinary singularities
  $x_1, x_2, \dotsc, x_s$ with multiplicities $r_1,r_2, \dotsc, r_s$.  If
  $\operatorname{Sec}(C)$ denotes the locus of secant lines to $C$, then the
  singular locus of $\operatorname{CH}_0(C)$ is
  $\operatorname{Sec}(C) \cup \bigcup_{i=1}^s \{ L \in \Gr(1,\PP^3) : x_i \in
  L \}$, the bidegree\index{bidegree} of $\operatorname{Sec}(C)$ is
  \[
    \Bigl( \tfrac{1}{2}(d-1)(d-2) - g - \textstyle\sum\limits_{i=1}^s
    \tfrac{1}{2}r_i (r_i-1), \tfrac{1}{2}d(d-1) \Bigr) \, ,
  \]
  and the singular locus of $\operatorname{Sec}(C)$, when $C$ is smooth,
  consists of all lines that intersect $C$ with total multiplicity at least
  $3$.

  Let $S \subset \PP^3$ be a general surface of degree $d$ with $d \geq 4$.
  If $\operatorname{Bit}(S)$ denotes the locus of bitangents\index{bitangent}
  to $S$ and $\operatorname{Infl}(S)$ denotes the locus of inflectional
  tangents to $S$, then the singular locus of $\operatorname{CH}_1(S)$ is
  $\operatorname{Bit}(S) \cup \operatorname{Infl}(S)$, the bidegree of
  $\operatorname{Bit}(S)$ is
  \[
    \bigl( \tfrac{1}{2}d(d-1)(d-2)(d-3), \tfrac{1}{2}d(d-2)(d-3)(d+3) \bigr)
    \, ,
  \] 
  the bidegree of $\operatorname{Infl}(S)$ is
  $\bigl( d(d-1)(d-2), 3d(d-2) \bigr)$, and the singular locus of
  $\operatorname{Infl}(S)$ consists of all lines that are inflectional
  tangents at at least two points of $S$ or intersect $S$ with multiplicity at
  least $4$ at some point.
\end{theorem}
The bidegree of $\operatorname{Infl}(S)$ also appears in
\cite[Prop.~4.1]{05petitjean}. The bidegrees of $\operatorname{Bit}(S)$,
$\operatorname{Infl}(S)$, and $\operatorname{Sec}(C)$, for smooth $C$, already
appear in \cite{05arrondo}, a paper to which we owe a great
debt. Nevertheless, we give new, more geometric, proofs not relying on Chern
class techniques.  The singular loci of $\operatorname{Sec}(C)$,
$\operatorname{Bit}(S)$, and $\operatorname{Infl}(S)$ are partially described
in Lemma~2.3, Lemma~4.3, and Lemma~4.6 in \cite{05arrondo}.

Using duality, we establish some relationships of the varieties in
Theorem~\ref{05mainthm}.

\begin{theorem}
  \label{05thm:second}
  If $C$ is a nondegenerate smooth space curve, then the secant lines of $C$
  are dual to the bitangent lines of the dual surface $C^\vee$ and the tangent
  lines of $C$ are dual to the inflectional tangent lines of $C^\vee$.
\end{theorem}

Congruences and line complexes have been actively studied both in the 19th
century and in modern times. The study of congruences goes back to
Kummer~\cite{05kummer}, who classified those of order $1$; the order of a
congruence is the number of lines in the congruence that pass through a
general point in $\PP^3$.  The Chow hypersurfaces of space curves were
introduced by Cayley~\cite{05cayley} and generalized to arbitrary varieties by
Chow and van der Waerden~\cite{05chowVDW}.  Many results from the second half
of the 19th century are detailed in Jessop's monograph~\cite{05jessop}.
Hurwitz hypersurfaces and further generalizations known as higher associated
or coisotropic hypersurfaces are studied in \cite{05gkz, 05coisotropic,
  05hurwitz}.  Catanese~\cite{05catanese} shows that Chow
hypersurfaces\index{Chow!hypersurface} of space curves and Hurwitz
hypersurfaces of surfaces are exactly the self-dual hypersurfaces in the
Grassmannian $\Gr(1,\PP^3)$.  Ran~\cite{05ran} studies surfaces of order $1$
in general Grassmannians and gives a modern proof of Kummer's classification.
Congruences play a role in algebraic vision and multi-view geometry, where
cameras are modeled as maps from $\PP^3$ to congruences~\cite{05bernd}.  The
multidegree of the image of several of those cameras is computed by Escobar
and Knutson in~\cite{05multidegree}.

In Sect.~\ref{05sec:grass}, we collect basic facts about the Grassmannian
$\Gr(1,\PP^3)$ and its subvarieties. Section~\ref{05sec:bisec} studies the
singular locus of the Chow hypersurface of a space curve and computes its
bidegree.  Section~\ref{05sec:bitang} describes the singular locus of a
Hurwitz hypersurface and Sect.~\ref{05sec:dual} uses projective
duality\index{projective~duality} to calculate the bidegree of its components.
In Sect.~\ref{05sec:intersec}, we connect the intersection theory in
$\Gr(1,\PP^3)$ to Chow and Hurwitz hypersurfaces.  Finally,
Section~\ref{05sec:singLoci} analyzes the singular loci of secant, bitangent,
and inflectional congruences.

This article provides complete solutions to Problem~5 on Curves, Problem~4 on
Surfaces, and Problem~3 on Grassmannians in \cite{05Sturmfels}.

\section{The Degree of a Subvariety in $\Gr(1,\PP^3)$}
\label{05sec:grass}

In this section, we provide the geometric definition for the degree of a
subvariety in $\Gr(1,\PP^3)$.  An alternative approach, using coefficients of
classes in the Chow ring, can be found in Sect.~\ref{05sec:intersec}.  For
information about subvarieties of more general Grassmannians, we
recommend~\cite{05subvarieties}.

The Grassmannian\index{Grassmannian} $\Gr(1,\PP^3)$ of lines in $\PP^3$ is a
$4$-dimensional variety that embeds into $\PP^5$ via the Pl\"ucker
embedding\index{Plucker@Pl\"ucker!embedding}.  In particular, the line in
$3$-space spanned by the distinct points
$(x_0 : x_1 : x_2 : x_3), (y_0 : y_1 : y_2 : y_3) \in \PP^3$ is identified
with the point
$(p_{0,1} : p_{0,2} : p_{0,3} : p_{1,2} : p_{1,3} : p_{2,3}) \in \PP^5$, where
$p_{i,j}$ is the minor formed of $i$th and $j$th columns of the matrix $\left[
  \begin{smallmatrix} 
    x_0 & x_1 & x_2 & x_3 \\ 
    y_0 & y_1 & y_2 & y_3 
  \end{smallmatrix} 
\right]$.  The Pl\"ucker coordinates\index{Plucker@Pl\"ucker!coordinates}
$p_{i,j}$ satisfy the relation
$p_{0,1} p_{2,3} - p_{0,2} p_{1,3} + p_{0,3} p_{1,2} = 0$. Moreover, every
point in $\PP^5$ satisfying this relation is the Pl\"ucker coordinates of some
line. Dually, a line in $\PP^3$ is the intersection of two distinct planes. If
the planes are given by the equations
$a_0 x_0 + a_1 x_1 + a_2 x_2 + a_3 x_3 = 0$ and
$b_0 x_0 + b_1 x_1 + b_2 x_2 + b_3 x_3 = 0$, then the minors $q_{i,j}$ of the
matrix $\left[
  \begin{smallmatrix} 
    a_0 & a_1 & a_2 & a_3 \\ 
    b_0 & b_1 & b_2 & b_3 
  \end{smallmatrix} 
\right]$ are the dual Pl\"ucker coordinates and also satisfy
$q_{0,1} q_{2,3} - q_{0,2} q_{1,3} + q_{0,3} q_{1,2} = 0$.  The map given by
$p_{0,1} \mapsto q_{2,3}$, $p_{0,2} \mapsto -q_{1,3}$,
$p_{0,3} \mapsto q_{1,2}$, $p_{1,2} \mapsto q_{0,3}$,
$p_{1,3} \mapsto -q_{0,2}$, and $p_{2,3} \mapsto q_{0,1}$ allows one to
conveniently pass between these two coordinate systems.

A \emph{line complex} is a threefold $\Sigma \subset \Gr(1,\PP^3)$.  For a
general plane $H \subset \PP^3$ and a general point $v \in H$, the degree of
$\Sigma$ is the number of points in $\Sigma$ corresponding to a line
$L \subset \PP^3$ such that $v \in L \subset H$.  For instance, if
$C \subset \PP^3$ is a curve, then the Chow
hypersurface\index{Chow!hypersurface}
$\operatorname{CH}_0(C) := \{ L \in \Gr(1,\PP^3) : C \cap L \neq \varnothing
\}$ is a line complex.  A general plane $H$ intersects $C$ in $\deg(C)$ many
points, so there are $\deg(C)$ many lines in $H$ that pass through a general
point $v \in H$ and intersect $C$; see Fig.~\ref{05fig:chowDeg}.
\begin{figure}[ht]
  \centering
  \includegraphics[width=10cm]{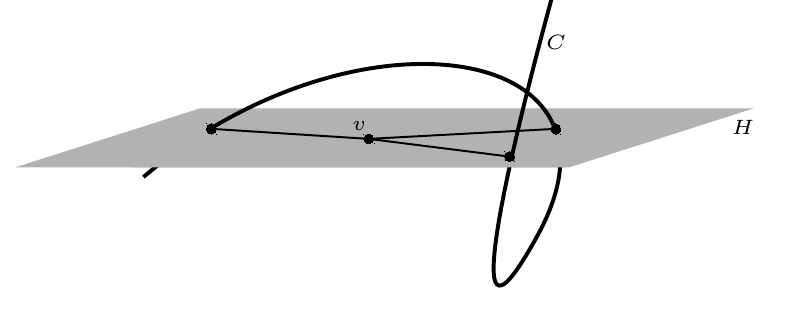}
  \caption{The degree of the Chow hypersurface}
  \label{05fig:chowDeg}
\end{figure}
Hence, the degree of the Chow hypersurface is equal to the degree of the
curve.

A \emph{congruence}\index{congruence} is a surface
$\Sigma \subset \Gr(1,\PP^3)$. For a general point $v \in \PP^3$ and a general
plane $H \subset \PP^3$, the bidegree\index{bidegree} of a congruence is a
pair $(\alpha, \beta)$, where the \emph{order} $\alpha$ is the number of
points in $\Sigma$ corresponding to a line $L \subset \PP^3$ such that
$v \in L$ and the \emph{class} $\beta$ is the number of points in $\Sigma$
corresponding to lines $L \subset \PP^3$ such that $L \subset H$.  For
instance, consider the congruence of all lines passing through a fixed point
$x$.  Given a general point $v$, this congruence contains a unique line
passing through $v$, namely the line spanned by $x$ and $v$.  Given a general
plane $H$, we have $x \not\in H$, so this congruence does not contain any line
that lies in $H$.  Hence, the set of lines passing through a fixed point is a
congruence with bidegree $(1,0)$.  A similar argument shows that the
congruence of lines lying in a fixed plane has bidegree $(0,1)$.

The degree of a curve $\Sigma \subset \Gr(1,\PP^3)$ is the number of points in
$\Sigma$ corresponding to a line $L \subset \PP^3$ that intersects a general
line in $\PP^3$.  Equivalently, it is the number of points in the intersection
of $\Sigma$ with the Chow hypersurface\index{Chow!hypersurface} of a general
line.  For instance, the set of all lines in $\PP^3$ that lie in a fixed plane
$H \subset \PP^3$ and contain a fixed point $v \in H$ forms a curve in
$\Gr(1, \PP^3)$.  This curve has degree $1$, because a general line has a
unique intersection point with $H$ and there is a unique line passing through
this point and $v$.  In other words, this curve is a line in $\Gr(1, \PP^3)$.

Finally, the degree of a zero-dimensional subvariety is simply the number of
points in the variety.

\section{Secants of Space Curves}
\label{05sec:bisec}

This section describes the singular locus of the Chow hypersurface of a space
curve.  For a curve with mild singularities, we also compute the bidegree of
its secant congruence.

A curve $C \subset \PP^3$ is defined by at least two homogeneous polynomials
in the coordinate ring of $\PP^3$, and these polynomials are not uniquely
determined.  However, there is a single equation that encodes the curve
$C$. Specifically, its Chow hypersurface\index{Chow!hypersurface|emph}
$\operatorname{CH}_0(C) := \{ L \in \Gr(1, \PP^3) : C \cap L \neq \varnothing
\}$ is determined by a single polynomial in the Pl\"ucker
coordinates\index{Plucker@Pl\"ucker!coordinates} on $\Gr(1, \PP^3)$.  This
equation, known as the \emph{Chow form}\index{Chow!form} of $C$, is unique up
to rescaling and the Pl\"ucker relation\index{Plucker@Pl\"ucker!relation}.
For more on Chow forms; see~\cite{05chow}.

\begin{example}[{{\normalfont {\cite[Prop.~1.2]{05chow}}}}]
  The twisted cubic is a smooth rational curve of degree $3$ in $\PP^3$.
  Parametrically, this curve is the image of the map
  $\nu_3 \colon \PP^1 \to \PP^3$ defined by
  $(s:t) \mapsto (s^3:s^2t:st^2:t^3)$.  The line $L$, which is determined by
  the two equations $a_0 x_0 + a_1 x_1 + a_2 x_2 + a_3 x_3 = 0$ and
  $b_0 x_0 + b_1 x_1 + b_2 x_2 + b_3 x_3 = 0$, intersects the twisted cubic if
  and only if there exists a point $(s:t) \in \PP^1$ such that
  \[
    a_0 s^3 + a_1 s^2t + a_2st^2 + a_3t^3 = 0 = b_0 s^3 + b_1 s^2t + b_2 st^2+
    b_3 t^3 \, .
  \]
  The resultant\index{resultant} for these two cubic polynomials, which can be
  expressed as a determinant of an appropriate matrix with entries in
  $\ZZ[a_0, a_1, a_2, a_3, b_0, b_1, b_2, b_3]$, vanishes exactly when they
  have a common root.  It follows that the line $L$ meets the twisted cubic if
  and only if
  \[
    0 = \det 
    \begin{bmatrix}
      a_0 & a_1 & a_2 & a_3 & 0 & 0 \\
      0 & a_0 & a_1 & a_2 & a_3 & 0 \\
      0 & 0 & a_0 & a_1 & a_2 & a_3 \\
      b_0 & b_1 & b_2 & b_3 & 0 & 0 \\
      0 & b_0 & b_1 & b_2 & b_3 & 0 \\
      0 & 0 & b_0 & b_1 & b_2 & b_3 
    \end{bmatrix} 
    = - \det
    \begin{bmatrix}
      q_{0,1} & q_{0,2}          & q_{0,3} \\
      q_{0,2} & q_{0,3} + q_{1,2} & q_{1,3} \\
      q_{0,3} & q_{1,3}          & q_{2,3}
    \end{bmatrix} \, ,
  \]
  where $q_{i,j}$ are the dual Pl\"ucker coordinates.  Hence, the Chow
  form\index{Chow!form} of the twisted cubic is
  $q_{0,3}^3 + q_{0,3}^2 q_{1,2}^{} - 2 q_{0,2}^{} q_{0,3}^{} q_{1,3}^{} +
  q_{0,1}^{} q_{1,3}^2 + q_{0,2}^2 q_{2,3}^{} - q_{0,1}^{} q_{0,3}^{}
  q_{2,3}^{} - q_{0,1}^{} q_{1,2}^{} q_{2,3}^{}$.
\end{example}

We next record a technical lemma.  If $I_X$ is the saturated homogeneous ideal
defining the subvariety $X \subset \PP^n$, then the tangent space $T_x(X)$ at
the point $x \in X$ can be identified with
$\bigl\{ y \in \PP^n :
\text{$\textstyle\sum\nolimits_{i=0}^n \tfrac{\partial f}{\partial x_i}(x) y_i
  = 0$ for all $f(x_0,x_1, \dotsc, x_n) \in I_X$} \bigr\}$.

\begin{lemma}
  \label{05prop:key}
  Let $f \colon X \to Y$ be a birational finite surjective morphism between
  irreducible projective varieties and let $y \in Y$.  The variety $Y$ is
  smooth at the point $y$ if and only if the fibre $f^{-1}(y)$ contains
  exactly one point $x \in X$, the variety $X$ is smooth at the point $x$, and
  the differential $d_x f \colon T_x (X) \to T_y(Y)$ is an injection.
\end{lemma}

\begin{proof}
  First, suppose that $Y$ is smooth at the point $y$.  Since $Y$ is normal at
  the point $y$, the Zariski Connectedness
  Theorem~\cite[Sect.~III.9.V]{05mumford}
  \index{Zariski~Connectedness~Theorem} proves that the fibre $f^{-1}(y)$ is a
  connected set in the Zariski topology.  As $f$ is a finite morphism, its
  fibres are finite and we deduce that $f^{-1}(y) = \{x \}$.  If $Y_0$ is the
  open set of smooth points in $Y$ and let $X_0 := f^{-1}(Y_0)$, then
  Zariski's Main Theorem~\cite[Sect.~III.9.I]{05mumford} implies that the
  restriction of $f$ to $X_0$ is an isomorphism of $X_0$ with $Y_0$.  In
  particular, we have that $x \in X_0 \subset X$ is a smooth point.  Moreover,
  Theorem~14.9 in \cite{05harris} shows that the differential $d_xf$ is
  injective.

  For the other direction, suppose that $f^{-1}(y) = \{ x \}$ for some smooth
  point $x \in X$ with injective differential $d_xf$.  Let $Y_1$ be an open
  neighbourhood of $y$ containing points in $Y$ with one-element fibres and
  injective differentials.  Combining Lemma~14.8 and Theorem~14.9 in
  \cite{05harris} produces an isomorphism of $X_1 := f^{-1}(Y_1)$ with $Y_1$.
  Since $x \in X_1$ is smooth, we conclude that $y \in Y_1 \subset Y$ is
  smooth.  \qed
\end{proof}

When the curve $C$ has degree at least two, the set of lines that meet it in
two points forms a surface $\operatorname{Sec}(C) \subset \Gr(1,\PP^3)$ called
the \emph{secant congruence}\index{congruence} of $C$. More precisely,
$\operatorname{Sec}(C)$ is the closure in $\Gr(1,\PP^3)$ of the set of points
corresponding to a line in $\PP^3$ which intersects the curve $C$ at two
smooth points.  A line meeting $C$ at a singular point might not belong to
$\operatorname{Sec}(C)$, even though it has intersection multiplicity at least
two with the curve; see Remark~\ref{05rem:singChow}.

The following theorem is the main result in this section.

\begin{theorem}
  \label{05thm:singChow}
  Let $C \subset \PP^3$ be an irreducible curve\index{curve!space} of degree
  at least $2$.  If $\operatorname{Sing}(C)$ denotes the singular locus of the
  curve $C$, then the singular locus of the Chow
  hypersurface\index{Chow!hypersurface} for $C$ is
  $\operatorname{Sec}(C) \cup \bigl( \bigcup_{x \in \operatorname{Sing}(C)} \{
  L \in \Gr(1,\PP^3) : x \in L \} \bigr)$.
\end{theorem}

\begin{proof}
  We first show that the incidence variety
  $\Phi_C := \{ (v, L) : v \in L \} \subset C \times \Gr(1,\PP^3)$ is smooth
  at the point $(v, L)$ if and only if the curve $C$ is smooth at the point
  $v \in C$.  Let $f_1, f_2, \dotsc, f_k \in \CC[x_0,x_1,x_2,x_3]$ be
  generators for the saturated homogeneous ideal of $C$ in $\PP^3$. Consider
  the affine chart of $\PP^3 \times \Gr(1, \PP^3)$ where $x_0 \neq 0$ and
  $p_{0,1} \neq 0$. We may assume that $v = (1 : \alpha : \beta : \gamma)$ and
  the line $L$ is spanned by the points $(1 : 0 : a : b)$ and
  $(0 : 1 : c : d)$.  We have that $v \in L$ if and only if the line $L$ is
  given by the row space of matrix
  \[
    \begin{bmatrix}
      1 & \alpha & \beta & \gamma \\ 0 & 1 & c & d
    \end{bmatrix}
    = 
    \begin{bmatrix}
      1 & \alpha  \\ 0 & 1 
    \end{bmatrix}
    \begin{bmatrix}
      1 & 0 & \beta-\alpha c & \gamma-\alpha d \\ 0 & 1 & c & d
    \end{bmatrix} \, ,
  \]
  which is equivalent to $a = \beta - \alpha c$ and $b = \gamma - \alpha d$.
  Hence, in the chosen affine chart, $\Phi_C$ can be written as
  \[
    \bigl\{ (\alpha, \beta, \gamma, a,b,c,d) :
    \text{$f_i(1,\alpha, \beta, \gamma) = 0$ for $1 \leq i \leq k$,
      $a = \beta - \alpha c$, $b = \gamma - \alpha d$} \bigr\} \, .
  \]
  As $\dim \Phi_C = 3$, it is smooth at the point $(v, L)$ if and only if its
  tangent space has dimension three or, equivalently, the Jacobian matrix
  \[
    \begin{bmatrix}
      \frac{\partial f_1}{\partial x_1}(1, \alpha, \beta, \gamma) &
      \frac{\partial f_1}{\partial x_2}(1, \alpha, \beta, \gamma) &
      \frac{\partial f_1}{\partial x_3}(1, \alpha, \beta, \gamma) & 0 & 0 &
      0 & 0 \\
      \frac{\partial f_2}{\partial x_1}(1, \alpha, \beta, \gamma) &
      \frac{\partial f_2}{\partial x_2}(1, \alpha, \beta, \gamma) &
      \frac{\partial f_2}{\partial x_3}(1, \alpha, \beta, \gamma) & 0 & 0 &
      0 & 0 \\
      \vdots & \vdots & \vdots & \vdots & \vdots & \vdots & \vdots \\
      \frac{\partial f_k}{\partial x_1}(1, \alpha, \beta, \gamma) &
      \frac{\partial f_k}{\partial x_2}(1, \alpha, \beta, \gamma) &
      \frac{\partial f_k}{\partial x_3}(1, \alpha, \beta, \gamma) & 0 & 0 &
      0 & 0 \\
      -c & 1 & 0 & -1 & 0 & - \alpha & 0 \\
      -d & 0 & 1 & 0 & -1 & 0 & - \alpha \\
    \end{bmatrix} 
  \]
  has rank four.  We see that this Jacobian matrix has rank four if and only
  if the Jacobian matrix of $C$ has rank two, in which case $v \in C$ is smooth.
  Therefore, we deduce that $\Phi_C$ is smooth at the point $(v,L)$ exactly
  when $C$ is smooth at the point $v$.

  By Lemma~14.8 in \cite{05harris}, the projection
  $\pi \colon \Phi_C \rightarrow \operatorname{CH}_0(C)$ defined by
  $(v, L) \mapsto L$ is finite; otherwise $C$ would contain a line
  contradicting our assumptions. Moreover, the general fibre of $\pi$ has
  cardinality $1$ because the general line $L \in \operatorname{CH}_0(C)$
  intersects $C$ in a single point. Hence, $\pi$ is birational.  Applying
  Lemma~\ref{05prop:key} shows that $\operatorname{CH}_0(C)$ is smooth
  at $L$ if and only if $\pi^{-1}(L) = \{ (v,L) \}$ where $v \in C$ is a
  smooth point and the differential $d_{(v,L)} \pi$ is injective.  Using our
  chosen affine chart, we see that the differential $d_{(v,L)} \pi$ sends
  every element in the kernel of the Jacobian matrix to its last four
  coordinates.  This map is not injective if and only if the kernel contains
  an element of the form $\begin{bmatrix} 
    \ast & \ast & \ast & 0 & 0 & 0 & 0
  \end{bmatrix}^{\transpose} \neq 0$.  Such an element belongs to the kernel
  if and only if it is equal to $\begin{bmatrix}
    \lambda & c \lambda & d \lambda & 0 & 0 & 0 & 0 
  \end{bmatrix}^{\transpose}$ for some $\lambda \in \CC \setminus \{ 0 \}$ and
  \[
    \frac{\partial f_i}{\partial x_1}(1,\alpha, \beta, \gamma) +
    c\frac{\partial f_i}{\partial x_2}(1,\alpha, \beta, \gamma) +
    d\frac{\partial f_i}{\partial x_3}(1,\alpha, \beta, \gamma) = 0
  \]
  for all $1 \leq i \leq k$.  Hence, for a smooth point $v \in C$, the
  differential $d_{(v,L)} \pi$ is not injective if and only if $L$ is the
  tangent line of $C$ at $v$.  Since we have that $|\pi^{-1}(L)| = 1$ if and
  only if $L$ is not a secant line and all tangent lines to $C$ are contained
  in $\operatorname{Sec}(C)$, we conclude that $\operatorname{CH}_0(C)$ is
  smooth at $L$ if and only if $L \notin \operatorname{Sec}(C)$ and $L$ meets
  $C$ at a smooth point.  \qed
\end{proof}

\begin{remark}
  \label{05rem:singChow}
  Local computations show that the secant congruence of $C$ generally does not
  contain all lines through singular points of $C$.  To be more explicit, let
  $x \in C$ be an \emph{ordinary singularity}; the point $x$ is the
  intersection of $r$ branches of $C$ with $r \geq 2$, and the $r$ tangent
  lines of the branches at $x$ are pairwise different.  We claim that a line
  $L$ intersecting $C$ only at the point $x$ is contained in
  $\operatorname{Sec}(C)$ if and only if $L$ lies in a plane spanned by two of
  the $r$ tangent lines at $x$.  The union of all those lines forms the
  tangent star of $C$ at $x$; see \cite{05star1, 05star2}.

  Suppose that $x = (1:0:0:0)$ and $L \in \operatorname{Sec}(C)$ intersects
  the curve $C$ only at the point $x$.  The line $L$ must be the limit of a
  family of lines $L_t$ that intersect $C$ at two distinct smooth points.
  Without loss of generality, the line $L$ is not one of the tangent lines of
  the curve $C$ at the point $x$ and each line $L_t$ intersects at least two
  distinct branches of $C$.  Since there are only finitely many branches, we
  can also assume that each line $L_t$ in the family intersects the same two
  branches of the curve $C$.  These two branches are parametrized by
  $\bigl( 1 : f_1(s) : f_2(s) : f_3(s) \bigr)$ and
  $\bigl( 1 : g_1(s) : g_2(s) : g_3(s) \bigr)$ with $f_i(0) = 0 = g_j(0)$ for
  $1 \leq i, j \leq 3$.  It follows that tangent lines to these branches are
  spanned by $x$ and $\bigl( 1 : f_1'(0) : f_2'(0) : f_3'(0) \bigr)$ or
  $\bigl( 1 : g_1'(0) : g_2'(0) : g_3'(0) \bigr)$.  Parametrizing intersection
  points, we see that the line $L_t$ intersects the first branch at
  $\bigl(1 : f_1\bigl( \varphi(t) \bigr) : f_2\bigl( \varphi(t) \bigr) :
  f_3\bigl( \varphi(t) \bigr) \bigr)$ and the second branch at
  $\bigl( 1 : g_1 \bigl( \psi(t) \bigr) : g_2 \bigl( \psi(t) \bigr) : g_3
  \bigl( \psi(t) \bigr) \bigr)$ where $\varphi(0) = 0 = \psi(0)$.  Hence, the
  Pl\"ucker coordinates\index{Plucker@Pl\"ucker!coordinates} for $L_t$ are
  \[
    \left( \tfrac{g_1 (\psi(t)) - f_1(\varphi(t))}{t} :
      \tfrac{g_2(\psi(t)) - f_2(\varphi(t))}{t} : \dotsb :
      \tfrac{f_2(\varphi(t)) g_3(\psi(t)) - f_3(\varphi(t)) g_2(\psi(t))}{t}
    \right) \, .
  \]
  Taking the limit as $t \to 0$, we obtain the line $L$ with Pl\"ucker
  coordinates
  \[
    \bigl( g_1'(0)\psi'(0) - f_1'(0)\varphi'(0) : g_2'(0)\psi'(0) -
    f_2'(0)\varphi'(0) : \dotsb : 0 \bigr) \, .
  \]
  This line is spanned by the point $x$ and
  \[
    \bigl( 1 : g_1'(0) \psi'(0) \!-\! f_1'(0) \varphi'(0) : g_2'(0) \psi'(0)
    \!-\!  f_2'(0) \varphi'(0) : g_3'(0) \psi'(0) \!-\! f_3'(0) \varphi'(0)
    \bigr) ,
  \]
  so it lies in the plane spanned by the two tangent lines.  From this
  computation, we also see that all lines passing through $x$ and lying in the
  plane spanned by the tangent lines can be approximated by lines that
  intersect both of the branches at points different from $x$. For this, one
  need only choose $\varphi(t) = \lambda t$ and $\psi(t) = \mu t$ for all
  possible $\lambda, \mu \in \CC \setminus \{ 0 \}$.
\end{remark}

Using Chern classes\index{Chern!class}, Proposition~2.1 in \cite{05arrondo}
calculates the bidegree\index{bidegree} of the secant congruence of a smooth
curve.  We give a geometric description of this bidegree and extend it to
curves with ordinary singularities.

\begin{theorem}
  \label{05thm:bidegSecant}
  If $C \subset \PP^3$ is a nondegenerate irreducible curve\index{curve!space}
  of degree $d$ and genus $g$ having only ordinary singularities
  $x_1, x_2, \dotsc, x_s$ with multiplicities $r_1, r_2, \dotsc, r_s$, then
  the bidegree of the secant congruence $\operatorname{Sec}(C)$ is
  \[
  \Biggl( \binom{d-1}{2} - g - \sum\limits_{i=1}^s \binom{r_i}{2},
  \binom{d}{2} \Biggr) \, .
\]
\end{theorem}

\begin{proof}
  Let $H \subset \PP^3$ be a general plane. The intersection of $H$ with $C$
  consists of $d$ points. Any two of these points define a secant line lying
  in $H$; see Fig.~\ref{05fig:secClass}.
  \begin{figure}[ht]
    \centering
    \includegraphics[width=10cm]{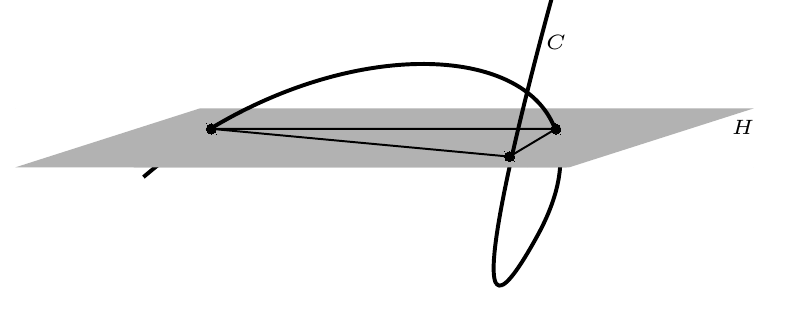}
    \caption{The class of the secant congruence}
    \label{05fig:secClass}
  \end{figure}
  Hence, there are $\binom{d}{2}$ secant lines contained in $H$, which gives
  the class of $\operatorname{Sec}(C)$.
 
  To compute the order of $\operatorname{Sec}(C)$, let $v \in \PP^3$ be a
  general point. Projecting away from $v$ defines a rational map
  $\pi_v \colon \PP^3 \dashrightarrow \PP^2$.  Set $C' := \pi_v(C)$. The map
  $\pi_v$ sends a line passing through $v$ and intersecting $C$ at two points
  to a simple node of the plane curve $C'$; see Fig.~\ref{05fig:proj}.
  Moreover, every ordinary singularity of $C$ is sent to an ordinary
  singularity of $C'$ with the same multiplicity, and the plane curve $C'$ has
  the same degree as the space curve $C$.  As the geometric genus is invariant
  under birational transformation, it also has the same genus; see
  \cite[Theorem~II.8.19]{05hartshorne}.  Thus, the genus-degree formula for
  plane curves~\cite[p.~54, Eq.~(7)]{05genusDegree} shows that the genus of
  $C$ is equal to $\binom{d-1}{2} - \sum_{i=1}^s \binom{r_i}{2}$ minus the
  number of secants of $C$ passing through $v$.  \qed
\end{proof}

\begin{remark}
  \label{05rem:bidegSecant}
  If $C \subset \PP^3$ is a curve of degree at least $2$ that is contained in
  a plane, then its secant congruence consists of all lines in that plane and
  has bidegree $(0,1)$.
\end{remark}

Problem~5 on Curves in~\cite{05Sturmfels} asks to compute the dimension and
bidegree of $\operatorname{Sing}(\operatorname{CH}_0(C))$.  When $C$ is not a
line, Theorem~\ref{05thm:singChow} establishes that
$\operatorname{Sing}(\operatorname{CH}_0(C))$ is $2$-dimensional. For
completeness, we also state its bidegree explicitly.

\begin{corollary}
  \label{05cor:curvebidegree}
  If $C \subset \PP^3$ is an irreducible curve of degree $d \geq 2$ and
  geometric genus $g$ having only ordinary singularities
  $x_1, x_2, \dotsc, x_s$ with multiplicities $r_1, r_2, \dotsc, r_s$, then
  the bidegree of $\operatorname{Sing} \bigl( \operatorname{CH}_0(C) \bigr)$
  equals
  $\Bigl( \tbinom{d-1}{2} - g - \textstyle\sum\limits_{i=1}^s \tbinom{r_i}{2}
  + s, \tbinom{d}{2} \Bigr)$ if $C$ is nondegenerate, and $(s,1)$ if $C$ is
  contained in a plane.
\end{corollary} 

\begin{proof}
  The bidegree\index{bidegree} of each congruence\index{congruence}
  $\{ L \in \Gr(1,\PP^3) : x_i \in L \}$ is $(1,0)$.  Hence, combining
  Theorem~\ref{05thm:singChow}, Theorem~\ref{05thm:bidegSecant}, and
  Remark~\ref{05rem:bidegSecant} proves the corollary.  \qed
\end{proof}

\section{Bitangents and Inflections of a Surface}
\label{05sec:bitang}

This section describes the singular locus of the Hurwitz hypersurface of a
surface in $\PP^3$.  For a surface $S \subset \PP^3$ that is not a plane, the
Hurwitz hypersurface\index{Hurwitz~hypersurface|emph} $\operatorname{CH}_1(S)$ is
the Zariski closure of the set of all lines in $\PP^3$ that are tangent to $S$
at a smooth point.  Its defining equation in Pl\"ucker
coordinates\index{Plucker@Pl\"ucker!coordinates} is known as the Hurwitz form
of $S$; see \cite{05hurwitz}.

In analogy with the secant congruence of a curve, we associate two congruences
to a surface $S \subset \PP^3$.  Specifically, the Zariski closure in
$\Gr(1,\PP^3)$ of the set of lines tangent to a surface $S$ at two smooth
points forms the \emph{bitangent congruence};
\[
  \operatorname{Bit}(S) := \overline{ \left\{ L \in \Gr(1,\PP^3) :
    \text{
      \begin{tabular}{p{6.8cm}}
        \small
        $x,y \in L \subset T_x(S) \cap T_y(S)$ for distinct smooth points
        $x, y \in S$
      \end{tabular}
    } \right\} } \, .
\]
The \emph{inflectional locus} associated to $S$ is the Zariski closure in
$\Gr(1,\PP^3)$ of the set of lines that intersect the surface $S$ at a smooth
point with multiplicity at least $3$;
\[
  \operatorname{Infl}(S) := \overline{ \left\{ L \in \Gr(1,\PP^3) : \text{
      \begin{tabular}{p{6.8cm}}
        \small
        $L$ intersects $S$ at a smooth point with multiplicity at
        least $3$
      \end{tabular}
    } \right\} } \, .
\]
A general surface of degree $d$ in $\PP^3$ is a surface defined by a
polynomial corresponding to a general point in
$\PP(\CC[x_0, x_1, x_2, x_3]_d)$.  For a general surface, the inflectional
locus is a congruence\index{congruence}.  However, this is not always the
case, as Remark~\ref{05rem:infcurve} demonstrates.

In parallel with Sect.~\ref{05sec:bisec}, the main result in this section
describes the singular locus of the Hurwitz
hypersurface\index{Hurwitz~hypersurface} of $S$.

\begin{theorem}
  \label{05thm:singHurwitz}
  If $S \subset \PP^3$ is an irreducible smooth surface\index{curve!space} of
  degree at least $4$ which does not contain any lines, then we have
  $\operatorname{Sing} \bigl( \operatorname{CH}_1(S) \bigr) =
  \operatorname{Bit}(S) \cup \operatorname{Infl}(S)$.
\end{theorem}

\begin{proof}
  We first show that the incidence variety
  \[
    \Phi_S := \{ (v, L) : v \in L \subset T_v(S) \} \subset S \times
    \Gr(1,\PP^3)
  \]
  is smooth.  Let $f \in \CC[x_0,x_1,x_2,x_3]$ be the defining equation for
  $S$ in $\PP^3$.  Consider the affine chart in $\PP^3 \times \Gr(1,\PP^3)$
  where $x_0 \neq 0$ and $p_{0,1} \neq 0$.  We may assume that
  $v = (1 : \alpha : \beta : \gamma)$ and the line $L$ is spanned by the
  points $(1 : 0 : a : b)$ and $(0 : 1 : c : d)$.  In this affine chart, $S$
  is defined by $g_0(x_1, x_2, x_3) := f(1, x_1, x_2, x_3)$.  As in the proof
  of Theorem~\ref{05thm:singChow}, we have that $v \in L$ if and only if
  $a = \beta - \alpha c$ and $b = \gamma - \alpha d$.  For such a pair
  $(v,L)$, we also have that $L \subset T_v(S)$ if and only if
  $(0:1:c:d) \in T_v(S)$.  Setting
  $g_1 := \frac{\partial g_0}{\partial x_1} + c \frac{\partial g_0}{\partial
    x_2} + d \frac{\partial g_0}{\partial x_3}$, we have $L \subset T_v(S)$ if
  and only if $g_1(\alpha, \beta, \gamma) = 0$.  Hence, in the chosen affine
  chart, $\Phi_S$ can be written as
  \[
    \bigl\{ (\alpha, \beta, \gamma, a,b,c,d) :
    \text{$g_j(\alpha, \beta, \gamma) = 0$ for $0 \leq j \leq 1$,
      $a = \beta -\alpha c$, $b = \gamma - \alpha d$ } \bigr\} \, .
  \]
  As $\dim \Phi_S = 3$, it is smooth at the point $(v,L)$ if and only if its
  tangent space has dimension three or, equivalently, its Jacobian matrix
  \[
    \begin{bmatrix}
      \frac{\partial g_0}{\partial x_1}(\alpha, \beta, \gamma) &
      \frac{\partial g_0}{\partial x_2}(\alpha, \beta, \gamma) &
      \frac{\partial g_0}{\partial x_3}(\alpha, \beta, \gamma) & 0 & 0 & 0 & 0
      \\
      \frac{\partial g_1}{\partial x_1}(\alpha, \beta, \gamma) &
      \frac{\partial g_1}{\partial x_2}(\alpha, \beta, \gamma) &
      \frac{\partial g_1}{\partial x_3}(\alpha, \beta, \gamma) & 0 & 0 &
      \frac{\partial g_0}{\partial x_2}(\alpha, \beta, \gamma) &
      \frac{\partial g_0}{\partial x_3}(\alpha, \beta, \gamma)
      \\
      -c & 1 & 0 & -1 & 0 & - \alpha & 0 \\
      -d & 0 & 1 & 0 & -1 & 0 & - \alpha
    \end{bmatrix} 
  \]
  has rank four.  Since $S$ is smooth, we deduce that this Jacobian matrix has
  rank four, so $\Phi_S$ is also smooth.

  Since $S$ does not contain any lines, all fibres of the projection
  $\pi \colon \Phi_S \rightarrow \operatorname{CH}_1(S)$ defined by
  $(v,L) \mapsto L$ are finite, so Lemma~14.8 in \cite{05harris} implies that
  $\pi$ is finite.  Moreover, the general fibre of $\pi$ has cardinality $1$,
  so $\pi$ is birational.  Applying Lemma~\ref{05prop:key} shows that
  $\operatorname{CH}_1(S)$ is smooth at the point $(v,L)$ if and only if the
  fibre $\pi^{-1}(L)$ consists of one point $(v,L)$ and the differential
  $d_{(x,L)}\pi$ is injective.  In particular, we have $|\pi^{-1}(L)| = 1$ if
  and only if $L$ is not a bitangent.  It remains to show that the
  differential $d_{(v,L)}\pi$ is injective if and only if $L$ is a simple
  tangent of $S$ at $v$.  Using our chosen affine chart, we see that the
  differential $d_{(v,L)}\pi$ projects every element in the kernel of the
  Jacobian matrix on its last four coordinates.  This map is not injective if
  and only if the kernel contains an element of the form
  $\begin{bmatrix} \ast & \ast & \ast & 0 & 0 & 0 & 0
  \end{bmatrix}^{\transpose} \neq 0$.  Such an element belongs to the kernel
  if and only if it is equal to $\begin{bmatrix}
    \lambda & c \lambda & d \lambda & 0 & 0 & 0 & 0
  \end{bmatrix}^{\transpose}$ for some $\lambda \in \CC \setminus \{ 0 \}$ and
  $g_1(\alpha, \beta, \gamma) = 0 = g_2 (\alpha, \beta, \gamma)$ where
  $g_2 := \frac{\partial g_1}{\partial x_1} + c \frac{\partial g_1}{\partial
    x_2} + d \frac{\partial g_1}{\partial x_3}$.  Parametrizing the line $L$
  by 
  \[
    \ell(s,t) := (s : s \alpha + t : s \beta + t c : s \gamma + t d)
  \]
  for $(s : t) \in \PP^1$ shows that the line $L$ intersects the surface $S$
  with multiplicity at least $3$ at $v$ if and only if
  $f \bigl( \ell(s,t) \bigr)$ is divisible by $t^3$.  This is equivalent to
  the conditions that
  $g_1(\alpha, \beta, \gamma) = \frac{\partial}{\partial t}\bigl[ f \bigl(
  \ell(s,t) \bigr) \bigr] \big|_{(1,0)} = 0$ and
  $g_2(\alpha, \beta, \gamma) = \frac{\partial^2}{\partial^2 t}\bigl[ f \bigl(
  \ell(s,t) \bigr) \bigr]\big|_{(1,0)} = 0$.  \qed
\end{proof}

\begin{remark}
  If $S$ is a surface of degree at most $3$ and the line $L$ is bitangent to
  $S$, then $L$ is contained in $S$. Indeed, if $L$ is not contained in $S$,
  then the intersection $L \cap S$ consists of at most $3$ points, counted
  with multiplicity, so $L$ cannot be a bitangent.  On the other hand, when
  the degree of $S$ is at least four, the hypothesis that $S$ does not contain
  any lines is relatively mild. For example, a general surface of degree at
  least $4$ in $\PP^3$ does not contain a line; see \cite{05vdWLines}.
\end{remark}

\section{Projective Duality}
\label{05sec:dual}

This section uses projective duality\index{projective~duality} to compute the
bidegrees of the components of the singular locus of the Hurwitz
hypersurface\index{Hurwitz~hypersurface} of a surface in $\PP^3$, and to
relate the secant congruence of a curve to the bitangent congruence of its
dual surface.

Let $\PP^n$ be the projectivization of the vector space $\CC^{n+1}$.  If
$(\PP^n)^*$ denotes the projectivization of the dual vector space
$(\CC^{n+1})^*$, then the points in $(\PP^n)^*$ correspond to hyperplanes in
$\PP^n$.  Given a projective subvariety $X \subset \PP^n$, a hyperplane in
$\PP^n$ is tangent to $X$ at a smooth point $x \in X$ if it contains the
embedded tangent space $T_x(X) \subset \PP^n$.  The \emph{dual
  variety}\index{variety!dual} $X^\vee$ is the Zariski closure in $(\PP^n)^*$
of the set of all hyperplanes in $\PP^n$ that are tangent to $X$ at some
smooth point.  

\begin{example}
  If $V$ is a linear subspace of $\CC^{n+1}$ and $X := \PP(V)$, then the dual
  variety $X^\vee$ is the set of all hyperplanes containing $\PP(V)$, which is
  exactly the projectivization of the orthogonal complement
  $V^{\perp} \subset (\CC^{n+1})^*$ with respect to the nondegenerate bilinear
  form $(x,y) \mapsto \sum_{i=0}^n x_i y_i$.  In particular, $X^\vee$ is not
  the projectivization of $V^*$, and $(\PP^n)^\vee = \varnothing$.
\end{example}

\begin{remark}
  The dual of a line in $\PP^2$ is a point, and the dual of a plane curve of
  degree at least $2$ is again a plane curve.  The dual of a line in $\PP^3$
  is a line, and the dual of a curve in $\PP^3$ of degree at least $2$ is a
  surface.  The dual of plane in $\PP^3$ is a point and the dual of a surface
  in $\PP^3$ of degree at least $2$ can be either a curve or a surface.
\end{remark}

From our perspective, the key properties of dual varieties are the following.
If $X$ is irreducible, then its dual $X^\vee$ is also irreducible; see
\cite[Proposition~I.1.3]{05gkz}.  Moreover, the Biduality Theorem shows that,
if $x \in X$ is smooth and $H \in X^\vee$ is smooth, then $H$ is tangent to
$X$ at the point $x$ if and only if the hyperplane in $(\PP^n)^*$
corresponding to $x$ is tangent to $X^\vee$ at the point $H$; see
\cite[Theorem~I.1.1]{05gkz}.  In particular, any irreducible variety
$X \subset \PP^n$ is equal to its double dual $(X^\vee)^\vee \subset \PP^n$;
again see \cite[Theorem~I.1.1]{05gkz}.

The next lemma, which relates the number and type of singularities of a plane
curve to the degree of its dual variety, plays an important role in
calculating the bidegrees of the bitangent and inflectional congruences.  A
point $v$ on a planar curve $C$ is a \emph{simple node} or a \emph{cusp} if
the formal completion of $\mathcal{O}_{C,v}$ is isomorphic to
$\CC[\![z_1, z_2]\!]/(z_1^2 + z_2^2)$ or $\CC[\![z_1, z_2]\!]/(z_1^3 + z_2^2)$
respectively; see Fig.~\ref{05fig:dual}.
\begin{figure}[ht]
  \centering
  \includegraphics[width=8cm]{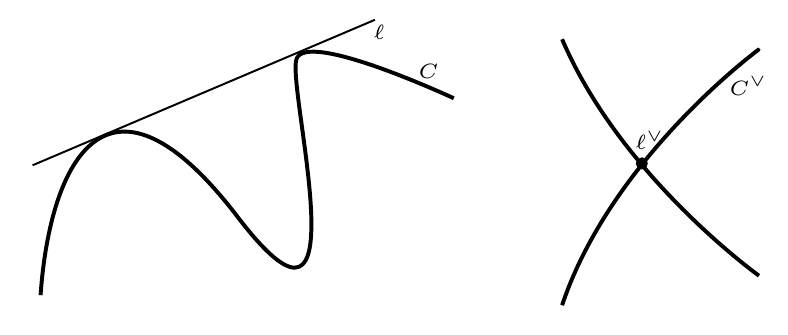}
  \includegraphics[width=8cm]{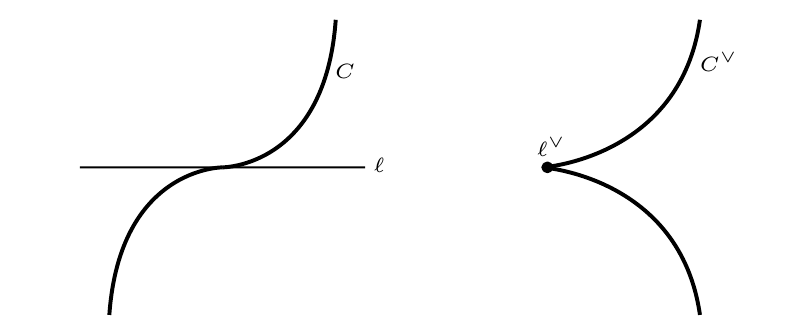}
  \caption{A bitangent and an inflectional line corresponding to a node
    and a cusp of the dual curve}
  \label{05fig:dual}
\end{figure}
Both singularities have multiplicity $2$; nodes have two distinct tangents and
cusps have a single tangent.

\begin{lemma}[{{\normalfont Pl\"ucker's
      formula\index{Plucker@Pl\"ucker!formula}
      \cite[Example~1.2.8]{05dolgachev}}}]
  \label{05lem:plucker}
  If $C \subset \PP^2$ is an irreducible curve of degree $d$ with exactly
  $\kappa$ cusps, $\delta$ simple nodes, and no other singularities, then the
  degree of the dual curve\index{curve!dual} $C^\vee$ is
  $d(d-1) - 3\kappa - 2 \delta$.
\end{lemma}

\begin{proof}[Sketch]
  Let $f \in \CC[x_0, x_1, x_2]$ be the defining equation for $C$ in $\PP^2$,
  so we have $\deg(f) = d$.  To begin, assume that $C$ is smooth.  The degree
  of its dual $C^\vee \subset (\PP^2)^*$ is the number of points of $C^\vee$
  lying on a general line $L \subset (\PP^2)^*$.  By duality, the degree
  equals the number of tangent lines to $C$ passing through a general point
  $y \in \PP^2$.  Such a tangent line at the point $v \in C$ passes through
  the point $y$ if and only if
  $g := y_0 \frac{\partial f}{\partial x_0}(v) + y_1 \frac{\partial
    f}{\partial x_1}(v) + y_2 \frac{\partial f}{\partial x_2}(v) = 0$.  Hence,
  the degree of $C^\vee$ is the number of points in $\variety(f,g)$; the
  vanishing set of $f$ and $g$.  Since $\deg(g) = d-1$, this finite set
  contains $d(d-1)$ points.

  If $C$ is singular, then the degree of $C^\vee$ is the number of lines that
  are tangent to $C$ at a smooth point and pass through the general point
  $y$. Those smooth points are contained in the set $\variety(f,g)$, but all
  of the singular points also lie in $\variety(f,g)$. The curve $\variety(g)$
  passes through each node of $C$ with intersection multiplicity two and
  through each cusp of $C$ with intersection multiplicity $3$.  Therefore,
  we conclude that $\deg(C^\vee) = d(d-1) - 3 \kappa -2 \delta$. \qed
\end{proof}

Using Lemma~\ref{05lem:plucker}, we can compute the degree of the Hurwitz
hypersurface of a smooth surface; this formula also follows from Theorem~1.1
in \cite{05hurwitz}.

\begin{proposition}
  \label{05prop:hurwDeg}
  For an irreducible smooth surface\index{surface!smooth} $S \subset \PP^3$ of
  degree $d$ with $d \geq 2$, the degree of the Hurwitz
  hypersurface\index{Hurwitz~hypersurface} $\operatorname{CH}_1(S)$ is
  $d(d-1)$.
\end{proposition}
  
\begin{proof}
  Let $H \subset \PP^3$ be a general plane and $v \in H$ be a general point.
  The degree of $\operatorname{CH}_1(S)$ is the number of tangent lines $L$ to
  $S$ such that $v \in L \subset H$.  Bertini's
  Theorem~\cite[Theorem~17.16]{05harris} implies that the intersection
  $S \cap H$ is a smooth plane curve of degree $d$.  The degree of
  $\operatorname{CH}_1(S)$ is the number of tangent lines to $S \cap H$
  passing through the general point $v$; see Fig.~\ref{05fig:hurwDeg}.
  \begin{figure}[ht]
    \centering
    \includegraphics[width=7cm]{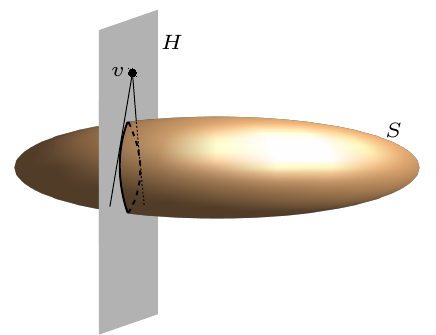}
    \caption{The degree of the Hurwitz hypersurface}
    \label{05fig:hurwDeg}
  \end{figure}
  By definition, this is equal to the degree of the dual plane curve
  $(S \cap H)^\vee$, so Lemma~\ref{05lem:plucker} shows 
  $\deg \bigl( \operatorname{CH}_1(S) \bigr) = d(d-1)$.  \qed
\end{proof}

Using Lemma~\ref{05lem:plucker}, we can also count the number of bitangents
and inflectional tangents to a general smooth plane curve.

\begin{proposition}
  \label{05prop:bitangents}
  A general smooth irreducible curve in $\PP^2$ of degree $d$ has exactly
  $\frac{1}{2}d(d-2)(d-3)(d+3)$ bitangents\index{bitangent} and $3d(d-2)$
  inflectional tangents.
\end{proposition}

\begin{proof}
  Let $C \subset \PP^2$ be a general smooth irreducible curve of degree $d$.
  A bitangent to $C$ corresponds to a node of $C^\vee$, and an inflectional
  tangent to $C$ corresponds to a cusp of $C^\vee$; see Fig.~\ref{05fig:dual}
  and \cite[pp.~277--278]{05GH}.  Lemma~\ref{05lem:plucker} shows that
  $C^\vee$ has degree $d(d-1)$.  Let $\kappa$ and $\delta$ be the number of
  cusps and nodes of $C^\vee$, respectively.  Applying
  Lemma~\ref{05lem:plucker} to the plane curve $C^\vee$ yields
  \[
    d = \deg(C) = \deg\bigl( (C^\vee)^\vee \bigr) = d(d-1) \bigl( d(d-1) -1
    \bigr) - 3 \kappa -2 \delta \, .
  \]  
  The dual curves $C$ and $C^\vee$ have the same geometric genus; see
  \cite[Proposition~1.5]{05tevelev}.  Hence, the genus-degree
  formula~\cite[p.~54, Eq.~(7)]{05genusDegree} gives
  \[
    \tfrac{1}{2}(d-1)(d-2) = \operatorname{genus}(C) =
    \operatorname{genus}(C^\vee) = \tfrac{1}{2} \bigl( d(d-1)-1 \bigr)
    \bigl( d(d-1)-2 \bigr) - \kappa -\delta \, .
  \]
  Solving this system of two linear equations in $\kappa$ and $\delta$, we
  obtain $\kappa = 3d(d-2)$ and $\delta = \frac{1}{2}d(d-2)(d-3)(d+3)$. \qed
\end{proof}

The next result is the main theorem in this section and solves Problem~4 on
Surfaces in~\cite{05Sturmfels}.  The bidegrees of the bitangent and the
inflectional congruence for a general smooth surface appear
in~\cite[Proposition~3.3]{05arrondo}, and the bidegree of the inflectional
congruence also appears in~\cite[Proposition~4.1]{05petitjean}.

\begin{theorem}
  \label{05thm:bidegTangent}
  Let $S \subset \PP^3$ be a general smooth irreducible
  surface\index{surface!smooth} of degree $d$ with $d \geq 4$. The
  bidegree\index{bidegree} of $\operatorname{Bit}(S)$ is
  $\bigl( \tfrac{1}{2}d(d-1)(d-2)(d-3), \tfrac{1}{2}d(d-2)(d-3)(d+3) \bigr)$,
  and the bidegree of $\operatorname{Infl}(S)$ is
  $\bigl( d(d-1)(d-2), 3d(d-2) \bigr)$.
\end{theorem}

\begin{proof}
  For a general plane $H \subset \PP^3$, Bertini's
  Theorem~\cite[Theorem~17.16]{05harris} implies that the intersection
  $S \cap H$ is a smooth plane curve of degree $d$. By
  Proposition~\ref{05prop:bitangents}, the number of bitangents to $S$
  contained in $H$ is $\frac{1}{2}d(d-2)(d-3)(d+3)$, which is the class of
  $\operatorname{Bit}(S)$.  Similarly, the number of inflectional tangents to
  $S$ contained in $H$ is $3d(d-2)$, which is the class of
  $\operatorname{Infl}(S)$.

  It remains to calculate the number of bitangents and inflectional lines of
  the surface $S$ that pass through a general point $y \in \PP^3$. Following
  the ideas in \cite[p.~230]{05piene}, let $f \in \CC[x_0,x_1,x_2,x_3]$ be the
  defining equation for $S$ in $\PP^3$, and consider the polar curve
  $C \subset S$ with respect to the point $y$; the set $C$ consists of all
  points $x \in S$ such that the line through $y$ and $x$ is tangent to $S$ at
  the point $x$; see Fig.~\ref{05fig:mainThm}.
  \begin{figure}[ht]
    \centering
    \includegraphics[width=10cm]{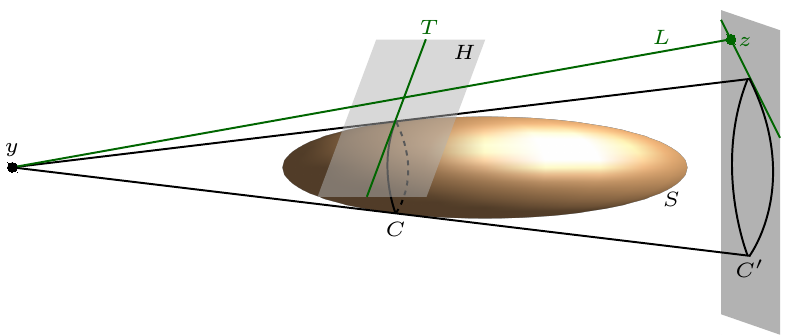}
    \caption{Polar curve}
    \label{05fig:mainThm}
  \end{figure}
  The condition that the point $x$ lies on the curve $C$ is equivalent to
  saying that the point $y$ belongs to $T_x(S)$.  As in the proof for
  Lemma~\ref{05lem:plucker}, we have $C = \variety(f,g)$ where
  $g := y_0 \frac{\partial f}{\partial x_0} + y_1 \frac{\partial f}{\partial
    x_1} + \dotsb + y_3 \frac{\partial f}{\partial x_3}$. Thus, the curve $C$
  has degree $d(d-1)$.

  Projecting away from the point $y$ gives the rational map
  $\pi_y \colon \PP^3 \dashrightarrow \PP^2$.  Restricted to the surface $S$,
  this map is generically finite, with fibres of cardinality $d$, and is
  ramified over the curve $C$. If $C'$ is the image of $C$ under $\pi_y$, then
  a bitangent to the surface $S$ that passes through $y$ contains two points
  of $C$ and these points are mapped to a simple node in $C'$; see
  Fig.~\ref{05fig:proj}.
  \begin{figure}[ht]
    \centering
    \includegraphics[width=11cm]{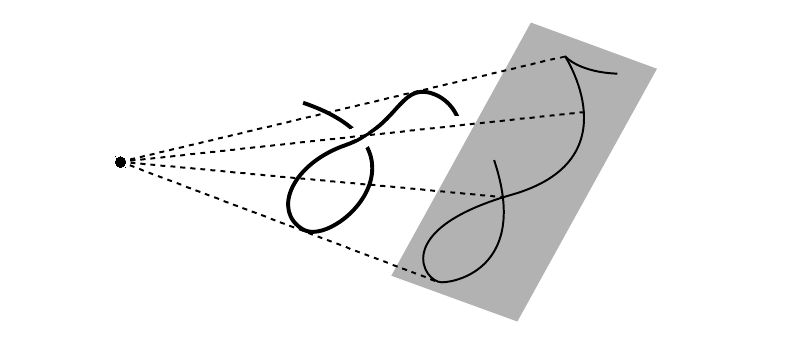}
    \caption{A secant projecting onto a node and a tangent projecting to a cusp}
    \label{05fig:proj}
  \end{figure}
  All of these nodes in $C'$ have two distinct tangent lines because no
  bitangent line passing through $y$ is contained in a bitangent plane that is
  tangent at the same two points as the line; the bitangent planes to $S$ form
  a $1$-dimensional family, so the union of bitangent lines they contain is a
  surface in $\PP^3$ that does not contain the general point $y$.

  We claim that the inflectional lines to $S$ passing through the point $y$
  are exactly the tangent lines of $C$ passing through $y$.  The line between
  a point $x \in S$ and the point $y$ is parametrized by the map
  $\ell \colon \PP^1 \to \PP^3$ which sends the point $(s:t) \in \PP^1$ to the
  point
  $(s x_0 + t y_0 : s x_1 + t y_1 : s x_2 + t y_2 : s x_3 + t y_3) \in \PP^3$.
  It follows that this line is an inflectional tangent to $S$ if and only if
  $f \bigl( \ell(s,t) \bigr)$ is divisible by $t^3$.  This is equivalent to
  the conditions that
  $\frac{\partial}{\partial t}\bigl[ f \bigl( \ell(s,t) \bigr) \bigr]
  \big|_{(1,0)} = 0$ and
  $\frac{\partial^2}{\partial t^2}\bigl[ f \bigl( \ell(s,t) \bigr) \bigr]
  \big|_{(1,0)} = 0$, which means that $x \in C$ and
  $y_0 \frac{\partial g}{\partial x_0} + y_1 \frac{\partial g}{\partial x_1} +
  \dotsb + y_3 \frac{\partial g}{\partial x_3} = 0$, or in other words
  $y \in T_x(C)$.  Therefore, the inflectional lines to $S$ passing through
  $y$ are the tangents to $C$ passing through $y$, and are mapped to the cusps
  of $C'$; again see Fig.~\ref{05fig:proj}.

  Since the bitangent and inflectional lines to $S$ passing through $y$
  correspond to nodes and cusps of $C'$, it suffices to count the number
  $\kappa'$ of cusps and the number $\delta'$ of simple nodes in the plane curve
  $C'$.  We subdivide these calculations as follows.
  \begin{description}
  \item[$\kappa' = d(d-1)(d-2)$:] From our parametrization of the line through
    points $x \in S$ and $y$, we see that this line is an inflectional tangent
    to $S$ if and only if $x \in \variety(f,g,h)$ where
    $h := y_0 \frac{\partial g}{\partial x_0} + y_1 \frac{\partial g}{\partial
      x_1} + \dotsb + y_3 \frac{\partial g}{\partial x_3}$. Since
    $\deg(h) = d-2$ and $S$ is general, the set $\variety(f,g,h)$ consists of
    $d(d-1)(d-2)$ points. \vspace*{5pt}
  \item[$\deg\bigl( (C')^\vee \bigr) = \deg(S^\vee)$:] By duality, the degree
    $d'$ of the curve $(C')^\vee$ is the number of tangent lines to
    $C' \subset \PP^2$ passing through a general point $z \in \PP^2$.  The
    preimage of $z$ under the projection $\pi_y$ is a line $L \subset \PP^3$
    containing $y$; see Fig.~\ref{05fig:mainThm}.  Hence, $d'$ is the number
    of tangent lines to $C$ intersecting $L$ in a point different from $y$.
    For every line $T$ that is tangent to $C$ at a point $x$ and intersects
    the line $L$, it follows that the pair $L$ and $T$ spans the tangent plane
    of $S$ at the point $x$.  On the other hand, given any plane $H$ which is
    tangent to $S$ at the point $x$ and contains $L$, we deduce that $x$ must
    lie on the polar curve $C$ and $H$ is spanned by $L$ and the tangent line
    to $C$ at $x$, so this tangent line intersects $L$.  Therefore, $d'$ is
    the number of tangent planes to $S$ containing $L$, which is the degree of
    the dual surface $S^\vee$. \vspace*{5pt}
  \item[$\deg(S^\vee) = d(d-1)^2$:] By duality, the degree of $S^\vee$ is the
    number of tangent planes to the surface $S$ containing a general line, or
    the number of tangent planes to $S$ containing two general points
    $y, z \in \PP^3$.  Thus, this is the number of intersection points of the
    two polar curves of $S$ determined by $y$ and $z$, which is the
    cardinality of the set $\variety(f,g,\tilde{g})$ where
    $\tilde{g} := z_0 \frac{\partial f}{\partial x_0} + z_1 \frac{\partial
      f}{\partial x_1} + \dotsb + z_3 \frac{\partial f}{\partial x_3}$.  Since
    $\deg(\tilde{g}) = d-1$, we conclude that $\deg(S^\vee) = d(d-1)^2$.
  \end{description}
  Finally, both the surface $S$ and the point $y$ are general, so
  Lemma~\ref{05lem:plucker} implies that
  $d(d-1)^2 = \deg \bigl( (C')^\vee \bigr) = \deg(C') \bigl( \deg(C') -1
  \bigr) - 3 d(d-1)(d-2) - 2 \delta'$.  Since $\deg(C') = \deg(C) = d(d-1)$,
  we have $\delta' = \frac{1}{2}d(d-1)(d-2)(d-3)$. \qed
\end{proof}

We end this section by proving that the secant locus of an irreducible smooth
curve is isomorphic to the bitangent congruence of its dual surface via the
natural isomorphism between $\Gr(1, \PP^3)$ and
$\Gr\bigl(1, (\PP^3)^* \bigr)$.  A subvariety $\Sigma \subset \Gr(1,\PP^3)$ is
sent under this isomorphism to the variety
$\Sigma^\perp \subset \Gr\bigl(1, (\PP^3)^* \bigr)$ consisting of the dual
lines $L^\vee$ for all $L \in \Sigma$.  For every congruence
$\Sigma \subset \Gr(1,\PP^3)$ with bidegree\index{bidegree} $(\alpha, \beta)$,
the bidegree of $\Sigma^\perp$ is $(\beta, \alpha)$.

\begin{theorem}
  \label{05thm:dual}
  If $C \subset \PP^3$ is a nondegenerate irreducible smooth
  curve\index{curve!space}, then we have
  $\operatorname{Sec}(C)^\perp = \operatorname{Bit}(C^\vee)$, the inflectional
  lines of $C^\vee$ are dual to the tangent lines of $C$, and
  $\operatorname{Infl}(C^\vee) \subset \operatorname{Bit}(C^\vee)$.
\end{theorem}

\begin{proof}
  We first show that
  $\operatorname{Sec}(C)^\perp = \operatorname{Bit}(C^\vee)$.  Consider a line
  $L$ that intersects $C$ at two distinct points $x$ and $y$, but is equal to
  neither $T_x(C)$ nor $T_y(C)$.  Together the line $L$ and $T_x(C)$ span a
  plane corresponding to a point $a \in C^\vee$.  Similarly, the span of the
  lines $L$ and $T_y(C)$ corresponds to a point $b \in C^\vee$.  Without loss
  of generality, we may assume that both $a$ and $b$ are smooth points in
  $C^\vee$.  By the Biduality Theorem, the points $a,b \in C^\vee$ must be
  distinct with tangent planes corresponding to $x$ and $y$.  Thus, the line
  $L^\vee$ is tangent to $C^\vee$ at the points $a$, $b$, and
  $\operatorname{Sec}(C)^\perp \subset \operatorname{Bit}(C^\vee)$.  To
  establish the other inclusion, let $L'$ be a line that is tangent to
  $C^\vee$ at two distinct smooth points $a,b \in C^\vee$.  The tangent planes
  at the points $a$, $b$ correspond to two points $x$, $y \in C$.  If
  $x \neq y$, then $(L')^{\vee}$ is the secant to $C$ through these two
  points.  If $x = y$, then the Biduality Theorem establishes that
  $(L')^{\vee}$ is the tangent line of $C$ at $x$.  In either case, we see
  that $\operatorname{Bit}(C^\vee) \subset \operatorname{Sec}(C)^\perp$, so
  $\operatorname{Sec}(C)^\perp = \operatorname{Bit}(C^\vee)$.

  For the second part, let $L$ be an inflectional line at a smooth point
  $a \in C^\vee$.  A point $y \in L^\vee \setminus C$ corresponds to a plane
  $H$ such that $L = T_a(C^\vee) \cap H$, so the line $L$ is also an
  inflectional line to the plane curve $C^\vee \cap H \subset H$.  Regarding
  $L$ as a subvariety of the projective plane $H$, its dual variety is a cusp
  on the plane curve $(C^\vee \cap H)^\vee \subset H^*$; see
  Fig.~\ref{05fig:dual}.  If
  $\pi_y \colon \PP^3 \dashrightarrow \PP^2 \cong H^*$ denotes the projection
  away from the point $y$, then we claim that $(C^\vee \cap H)^\vee$ equals
  $\pi_y(C)$; for a more general version see \cite[Proposition~6.1]{05holme}.
  Indeed, a smooth point $z \in \pi_y(C)$ is the projection of a point of $C$
  whose tangent line does not contain $y$.  Together this tangent line and the
  point $y$ span a plane such that its dual point $w$ is contained in the
  curve $C^\vee \cap H$.  Thus, the tangent line $T_z\bigl( \pi_y(C) \bigr)$
  equals $\pi_y(w^\vee)$; the latter is the line in $H^*$ dual to the point
  $w \in H$.  In other words, we have
  $\bigl( \pi_y(C) \bigr)^\vee \subset C^\vee \cap H$.  Since both curves are
  irreducible, this inclusion must be an equality.  Hence, when considering
  $L$ in the projective plane $H$, its dual point is a cusp of $\pi_y(C)$.  It
  follows that $L^\vee$ is the tangent line $T_x(C)$, where $x \in C$ is the
  point corresponding to the tangent plane $T_a(C^\vee)$; see
  Fig.~\ref{05fig:proj}.  Reversing these arguments shows that the dual of a
  tangent line to $C$ is an inflectional line to $C^\vee$.  Since every
  tangent line to $C$ is contained in $\operatorname{Sec}(C)$, we conclude
  that $\operatorname{Infl}(C^\vee) \subset \operatorname{Bit}(C^\vee)$.  \qed
\end{proof}

\subruninhead{Proof of Theorem~\ref{05thm:second}.}
This result is a restatement of Theorem~\ref{05thm:dual}. \qed

\begin{remark}
  \label{05rem:infcurve}
  Theorem~\ref{05thm:dual} shows that $\operatorname{Infl}(C^\vee)$ is a
  curve, as $\operatorname{Infl}(C^\vee)^\perp$ is the set of tangent lines to
  $C$, so the inflectional locus of a surface in $\PP^3$ is not always a
  congruence.
\end{remark}

\begin{remark}
  For a curve $C \subset \PP^3$ with dual surface $C^\vee \subset (\PP^3)^*$,
  Theorem~20 in \cite{05coisotropic} establishes that
  $\operatorname{CH}_0(C)^\perp = \operatorname{CH}_1(C^\vee)$.  Combined with
  Theorem~\ref{05thm:dual}, we see that the singular locus of the Hurwitz
  hypersurface\index{Hurwitz~hypersurface} $\operatorname{CH}_1(C^\vee)$, for
  smooth $C$, has just one component, namely the bitangent\index{bitangent}
  congruence.
\end{remark}

\begin{remark}
  For a surface $S \subset \PP^3$ with dual surface
  $S^\vee \subset (\PP^3)^*$, Theorem~20 in \cite{05coisotropic} also
  establishes that
  $\operatorname{CH}_1(S)^\perp = \operatorname{CH}_1(S^\vee)$.  If both $S$
  and $S^\vee$ have mild singularities, then the proof of Lemma~5.1 in
  \cite{05arrondo} shows that
  $\operatorname{Bit}(S)^\perp = \operatorname{Bit}(S^\vee)$.
\end{remark}

\section{Intersection Theory on $\Gr(1, \PP^3)$}
\label{05sec:intersec}

In this section, we recast the degree of a subvariety in $\Gr(1, \PP^3)$ in
terms of certain products in the Chow ring.

Consider a smooth irreducible variety $X$ of dimension $n$.  For each
$j \in \NN$, the group $Z^j(X)$ of codimension-$j$ cycles is the free abelian
group generated by the closed irreducible subvarieties of $X$ having
codimension $j$.  Given a variety $W$ of codimension $j-1$ and a nonzero
rational function $f$ on $W$, we have the cycle
$\operatorname{div}(f) := \sum_Z \operatorname{ord}_{Z}(f) \, Z$ where the sum
runs over all subvarieties $Z$ of $W$ with codimension $1$ in $W$ and
$\operatorname{ord}_{Z}(f) \in \ZZ$ is the order of vanishing of $f$ along
$Z$.  The group of cycles rationally equivalent to zero is the subgroup
generated by the cycles $\operatorname{div}(f)$ for all codimension-$(j-1)$
subvarieties $W$ of $X$ and all nonzero rational functions $f$ on $W$.  The
Chow group $A^j(X)$ is the quotient of $Z^j(X)$ by the subgroup of cycles
rationally equivalent to zero.  We typically write $[Z]$ for the class of a
subvariety $Z$ in the appropriate Chow group.  Since $X$ is the unique
subvariety of codimension $0$, we see that $A^0(X) \cong \ZZ$.  We also have
$A^1(X) \cong \Pic(X)$.  Crucially, the direct sum
$A^*(X) := \bigoplus_{j = 0}^n A^j(X)$ forms a commutative $\ZZ$-graded ring
called the \emph{Chow ring}\index{Chow!ring} of $X$.  The product arises from
intersecting cycles: for subvarieties $V$ and $W$ of $X$ having codimension
$j$ and $k$ and intersecting transversely, the product $[V][W] \in A^{j+k}(X)$
is the sum of the irreducible components of $V \cap W$.  More generally,
intersection theory aims to construct an explicit cycle to represent the
product $[V][W]$.

\begin{example}
  The Chow ring of $\PP^n$ is isomorphic to $\ZZ[H]/(H^{n+1})$ where $H$ is
  the class of a hyperplane.  In particular, any subvariety of codimension $d$
  is rationally equivalent to a multiple of the intersection of $d$
  hyperplanes.
\end{example}

To a given a vector bundle $\mathcal{E}$ of rank $r$ on $X$, we associate its
\emph{Chern classes}\index{Chern!class} $c_i(\mathcal{E}) \in A^i(X)$ for
$0 \leq i \leq r$; see \cite{05chern}.  When $\mathcal{E}$ is globally
generated, these classes are represented by degeneracy loci; the class
$c_{r+1-j}(\mathcal{E})$ is associated to the locus of points $x \in X$ where
$j$ general global sections of $\mathcal{E}$ fail to be linearly independent.
In particular, $c_r(\mathcal{E})$ is represented by the vanishing locus of a
single general global section.  Given a short exact sequence
$0 \to \mathcal{E}' \to \mathcal{E} \to \mathcal{E}'' \to 0$ of vector
bundles, the Whitney Sum Formula\index{Whitney~Sum~Formula} asserts that
$c_k(\mathcal{E}) = \sum_{i+j=k}c_i(\mathcal{E}') c_j(\mathcal{E}'')$; see
\cite[Theorem~3.2]{05fulton}.  Moreover, if
$\mathcal{E}^* := \mathcal{H}\!\textit{om}(\mathcal{E}, \mathcal{O}_X)$
denotes the dual vector bundle, then we have
$c_i(\mathcal{E}^*) = (-1)^i c_i(\mathcal{E})$ for $0 \leq i \leq r$; see
\cite[Remark~3.2.3]{05fulton}.

\begin{example}
  Given nonnegative integers $a_1, a_2, \dotsc, a_n$, consider the vector
  bundle
  $\mathcal{E} := \mathcal{O}_{\PP^n}(a_1) \oplus \mathcal{O}_{\PP^n}(a_2)
  \oplus \dotsb \oplus \mathcal{O}_{\PP^n}(a_n)$. Since each
  $\mathcal{O}_{\PP^n}(a_i)$ is globally generated, the Chern class
  $c_1 \bigl( \mathcal{O}_{\PP^n}(a_i) \bigr)$ is the vanishing locus of a
  general homogeneous polynomial $\CC[x_0,x_1, \dotsc, x_n]$ of degree $a_i$,
  so $c_1 \bigl( \mathcal{O}_{\PP^n}(a_i) \bigr) = a_iH$ in
  $A^*(\PP^n)$. Hence, the Whitney Sum Formula implies that
  $c_n(\mathcal{E}) = \prod_{i=1}^n c_1 \bigl( \mathcal{O}(a_i) \bigr) =
  \prod_{i=1}^n (a_i H)$.
\end{example}

\begin{example} 
  \label{05tangentBundle}
  If $\mathcal{T}_{\PP^n}$ is the tangent bundle on $\PP^n$, then we have the
  short exact sequence
  $0 \to \mathcal{O}_{\PP^n} \to \mathcal{O}_{\PP^n}(1)^{\oplus (n+1)} \to
  \mathcal{T}_{\PP^n} \to 0$; see \cite[Example~8.20.1]{05hartshorne}. The
  Whitney Sum Formula\index{Whitney~Sum~Formula} implies that
  $c_1(\mathcal{T}_{\PP^n}) = (n+1) c_1 \bigl( \mathcal{O}_{\PP^n}(1) \bigr) -
  c_1(\mathcal{O}_{\PP^n}) = (n+1)H$ and
  $c_2(\mathcal{T}_{\PP^n}) = c_2 \bigl( \mathcal{O}_{\PP^n}(1)^{\oplus(n+1)}
  \bigr) = \binom{n+1}{2}H^2$.
\end{example}

\begin{example} 
  \label{05tangentBundleSubvariety}
  Let $Y \subset \PP^n$ be a smooth hypersurface of degree $d$.  If
  $\mathcal{T}_{Y}$ is the tangent bundle of $Y$, then we have the exact
  sequence
  $0 \to \mathcal{T}_Y \to \mathcal{T}_{\PP^n}|_Y \to
  \mathcal{O}_{\PP^n}(d)|_Y \to 0$; see \cite[Proposition~8.20]{05hartshorne}.
  Setting $h := H|_Y$ in $A^*(Y)$, the Whitney Sum Formula implies that
  $c_1(\mathcal{T}_Y) = c_1(\mathcal{T}_{\PP^n}|_Y) - c_1\bigl(
  \mathcal{O}_{\PP^n}(d)|_Y \bigr) = (n+1)h - dh = (n+1-d)h$ and
  $c_2(\mathcal{T}_Y) = c_2(\mathcal{T}_{\PP^n}|_Y) - c_1(\mathcal{T}_{Y})
  c_1\bigl( \mathcal{O}_{\PP^n}(d)|_Y \bigr) = \bigl( \binom{n+1}{2} -
  (n+1-d)d \bigr) h^2$.
\end{example}

We next focus on the Chow ring of $\Gr(1,\PP^3)$; see \cite{05subvarieties,
  05chern}. Fix a complete flag $v_0 \in L_0 \subset H_0 \subset \PP^3$ where
the point $v_0$ lies in the line $L_0$, and the line $L_0$ is contained in the
plane $H_0$.  The Schubert varieties\index{Schubert!variety} in $\Gr(1,\PP^3)$
are the following subvarieties:
\begin{xalignat*}{2}
  \Sigma_0 &:= \Gr(1,\PP^3) \, , & 
  \Sigma_1 &:= \{ L : L \cap L_0 \neq \varnothing \} \subset \Gr(1,\PP^3) \, ,
  \\
  \Sigma_{1,1} &:= \{ L :  L \subset H_0 \} \subset \Gr(1,\PP^3) \, , &
  \Sigma_{2} &:= \{ L : v_0 \in L  \} \subset \Gr(1,\PP^3) \, , \\
  \Sigma_{2,1} &:= \{ L : v_0 \in L \subset H_0 \} \subset \Gr(1,\PP^3) \, , &
  \Sigma_{2,2} &:= \{ L_0 \} \subset \Gr(1,\PP^3) \, .
\end{xalignat*}
The corresponding classes $\sigma_I := [\Sigma_I]$, called the \emph{Schubert
  cycles}\index{Schubert!cycle}, form a basis for the Chow ring
$A^\ast(\Gr(1,\PP^3))$; see \cite[Theorem 5.26]{053264}.  Since the sum of the
subscripts gives the codimension, we have
\begin{xalignat*}{3}
  A^0 \bigl( \Gr(1, \PP^3) \bigr) &\cong \ZZ \sigma_0 \, , &
  A^1 \bigl( \Gr(1, \PP^3) \bigr) &\cong \ZZ \sigma_1 \, , &
  A^2 \bigl( \Gr(1, \PP^3) \bigr) &\cong \ZZ \sigma_{1,1} \oplus \ZZ \sigma_2
  \, ,\\
  A^3 \bigl( \Gr(1, \PP^3) \bigr) &\cong \ZZ \sigma_{2,1} \, , &
  A^4 \bigl( \Gr(1, \PP^3) \bigr) &\cong \ZZ \sigma_{2,2} \, .
\end{xalignat*}
To understand the product structure, we use the transitive action of
$\GL(4,\CC)$ on $\Gr(1,\PP^3)$.  Specifically, Kleiman's Transversality
Theorem~\cite{05kleiman} shows that, for two subvarieties $V$ and $W$ in
$\Gr(1,\PP^3)$, a general translate $U$ of $V$ under the
$\GL(4,\CC)$\nobreakdash-action is rationally equivalent to $V$ and the
intersection of $U$ and $W$ is transversal at the generic point of any
component of $U \cap W$.  Hence, we have $[V][W] = [U \cap W]$.  To determine
the product $\sigma_{1,1} \sigma_2$, we intersect general varieties
representing these classes: $\sigma_{1,1}$ consists of all lines $L$ contained
in a fixed plane $H_0$, and $\sigma_2$ is all lines $L$ containing a fixed
point $v_0$. Since a general point does not lie in a general plane, we see
that $\sigma_{1,1} \sigma_2 = 0$. Similar arguments yield all products:
\begin{xalignat*}{4} 
  \sigma_{1,1}^2 &= \sigma_{2,2} \, , &
  \sigma_2^2 &= \sigma_{2,2} \, , &
  \sigma_{1,1}\sigma_2 &= 0 \, , &
  \sigma_1 \sigma_{2,1} &= \sigma_{2,2} \, , \\
  \sigma_1 \sigma_{1,1} &= \sigma_{2,1} \, , &
  \sigma_1 \sigma_{2} &= \sigma_{2,1} \, , &
  \sigma_1^2 &=\sigma_2+\sigma_{1,1} \, .
\end{xalignat*}

The degree of a subvariety in $\Gr(1, \PP^3)$, introduced in
Sect.~\ref{05sec:grass}, can be interpreted as certain coefficients of its
class in the Chow ring\index{Chow!ring}.  Geometrically, the order $\alpha$ of
a surface $X \subset \Gr(1,\PP^3)$ is the number of lines in $X$ passing
through the general point $v_0$.  Since we may intersect $X$ with a general
variety representing $\sigma_2$, it follows that $\alpha$ equals the
coefficient of $\sigma_2$ in $[X]$.  Similarly, the class $\beta$ of $X$ is
the coefficient of $\sigma_{1,1}$ in $[X]$, the degree of a threefold
$\Sigma \subset \Gr(1,\PP^3)$ is the coefficient of $\sigma_{1}$ in
$[\Sigma]$, and the degree of a curve $C \subset \Gr(1,\PP^3)$ is the
coefficient of $\sigma_{2,1}$ in $[C]$.

The degree of a subvariety in $\Gr(1, \PP^3)$ also has a useful
reinterpretation via Chern classes of tautological vector bundles.  Let
$\mathcal{S}$ denote the tautological subbundle, the vector bundle whose fibre
over the point $W \in \Gr(1, \PP^3)$ is the $2$-dimensional vector space
$W \subseteq \CC^4$.  Similarly, let $\mathcal{Q}$ be the tautological
quotient bundle whose fibre over $W$ is $\CC^4/W$.  Both $\mathcal{S}^*$ and
$\mathcal{Q}$ are globally generated;
$H^0\bigl( \Gr(1,\PP^3), \mathcal{S}^\ast \bigr) \cong (\CC^4)^*$ and
$H^0 \bigl( \Gr(1,\PP^3), \mathcal{Q} \bigr) \cong \CC^4$; see
\cite[Proposition~0.5]{05subvarieties}.  A global section of $\mathcal{S}^*$
corresponds to a nonzero map $\varphi \colon \CC^4 \to \CC$, where its value
at the point $W$ is $\varphi|_W \colon W \to \CC$.  The Chern class
$c_2(\mathcal{S}^*)$ is represented by the vanishing locus of $\varphi$, so we
have $c_2(\mathcal{S}^*) = \sigma_{1,1} = c_2(\mathcal{S})$.  For two general
sections $\varphi, \psi \colon \CC^4 \to \CC$ of $\mathcal{S}^*$, the Chern
class\index{Chern!class} $c_1(\mathcal{S}^*)$ is represented by the locus of
points $W$ where $\varphi|_W$ and $\psi|_W$ fail to be linearly independent or
$W \cap \ker(\varphi) \cap \ker(\psi) \neq \{ 0 \}$.  Generality ensures that
$\ker(\varphi) \cap \ker(\psi)$ is a $2$-dimensional subspace of $\CC^4$, so
$c_1(\mathcal{S}^*) = - c_1(\mathcal{S}) = \sigma_1$.  Similarly, a global
section of $\mathcal{Q}$ corresponds to a point $v \in \CC^4$; its value at
$W$ is simply the image of the point in $\CC^4/W$.  Thus, $c_2(\mathcal{Q})$
is represented by the locus of those $W$ containing $v$, which is $\sigma_2$.
Two global sections of $\mathcal{Q}$ are linearly dependent at $W$ when the
$2$-dimensional subspace of $\CC^4$ spanned by the points intersects $W$
nontrivially, so $c_1(\mathcal{Q}) = \sigma_1$.  Finally, for a surface
$X \subset \Gr(1,\PP^3)$ with $[X] = \alpha \sigma_2 + \beta \sigma_{1,1}$, we
obtain
\begin{align*}
  c_2(\mathcal{Q}) \, [X] 
  &= \sigma_2 (\alpha \sigma_2 + \beta \sigma_{1,1}) = \alpha \sigma_{2,2} \,
  , \\
  c_2(\mathcal{S}) \, [X] 
  &= \sigma_{1,1} (\alpha \sigma_2 + \beta \sigma_{1,1}) = \beta \sigma_{2,2}
    \, ,
\end{align*}
so computing the bidegree\index{bidegree} is equivalent to calculating the
products $c_2(\mathcal{Q}) \, [X]$ and $c_2(\mathcal{S}) \, [X]$ in the Chow
ring.

We close this section with three examples demonstrating this approach.

\begin{example}
  Given a smooth surface $S$ in $\PP^3$, we recompute the degree of
  $\operatorname{CH}_1(S)$; compare with Proposition~\ref{05prop:hurwDeg}.
  Theorem~9 in \cite{05coisotropic} implies that this degree equals the degree
  $\delta_1(S)$ of the first polar locus
  $M_1(S) = \{ x \in S : y \in T_xS \}$, where $y$ is a general point of
  $\PP^3$ (this locus is the polar curve in the proof of
  Theorem~\ref{05thm:bidegTangent}).  Letting $T_S$ be the tangent bundle of
  $S$, Example~14.4.15 in \cite{05fulton} shows that
  $\delta_1(S) = \deg \bigl( 3h - c_1(T_S) \bigr)$.  Hence,
  Example~\ref{05tangentBundleSubvariety} gives
  $\delta_1(S) = \deg(3h-h(3+1-d))=(d-1) \deg(h)$. Since $S$ is a degree $d$
  surface, the degree of the hyperplane $h$ equals $d$, so
  $\delta_1(S) = d(d-1)$.
\end{example}

\begin{example}[{{\normalfont Problem~3 on Grassmannians
      in~\cite{05Sturmfels}}}]
  Let $S_1, S_2\subset \PP^3$ be general surfaces of degree $d_1$ and $d_2$,
  respectively, with $d_1, d_2 \geq 4$.  To find the number of lines bitangent
  to both surfaces, it suffices to compute the cardinality of
  $\operatorname{Bit}(S_1) \cap
  \operatorname{Bit}(S_2)$. Theorem~\ref{05thm:bidegTangent} establishes that,
  for all $1 \leq i \leq 2$, we have
  $[\operatorname{Bit}(S_i)] = \alpha_i \sigma_2 + \beta_i \sigma_{1,1}$ where
  $\alpha_i := \tfrac{1}{2}d_i(d_i-1)(d_i-2)(d_i-3)$ and
  $\beta_i := \tfrac{1}{2}d_i(d_i-2)(d_i-3)(d_i+3)$.  It follows that
  $[\operatorname{Bit}(S_1) \cap \operatorname{Bit}(S_2)] =
  [\operatorname{Bit}(S_1)][\operatorname{Bit}(S_2)] = (\alpha_1 \alpha_2 +
  \beta_1 \beta_2) \sigma_{2,2}$, so the number of lines bitangent to $S_1$
  and $S_2$ is
  \begin{multline*}
    \tfrac{1}{4}d_1(d_1-1)(d_1-2)(d_1-3) d_2(d_2-1)(d_2-2)(d_2-3) \\
    + \tfrac{1}{4} d_1(d_1-2)(d_1-3)(d_1+3) d_2(d_2-2)(d_2-3)(d_2+3) \, .
  \end{multline*}
\end{example}

\begin{example}
  Let $S \subset \PP^3$ be a general surface of degree $d_1$ with
  $d_1 \geq 4$, and let $C \subset \PP^3$ be a general curve of degree $d_2$
  and geometric genus $g$ with $d_2 \geq 2$. To find the number of lines
  bitangent to $S$ and secant to $C$, it suffices to compute the cardinality
  of $\operatorname{Bit}(S) \cap \operatorname{Sec}(C)$.
  Theorem~\ref{05thm:bidegTangent} and Theorem~\ref{05thm:bidegSecant} imply
  that
  \begin{align*}
    [\operatorname{Bit}(S)]
    &= \tfrac{1}{2}d_1(d_1-1)(d_1-2)(d_1-3) \, \sigma_2 +
      \tfrac{1}{2}d_1(d_1-2)(d_1-3)(d_1+3) \, \sigma_{1,1} \, , \\
    [\operatorname{Sec}(C)] 	 
    &= \bigl( \tfrac{1}{2}(d_2-1)(d_2-2)-g \bigr) \, \sigma_2
      + \tfrac{1}{2}d_2(d_2-1) \, \sigma_{1,1} \, .
  \end{align*}
  It follows that
  $[\operatorname{Bit}(S) \cap \operatorname{Sec}(C)] =
  [\operatorname{Bit}(S)][\operatorname{Sec}(C)] = \gamma \sigma_{2,2}$ where
  \begin{multline*}
    \gamma := \tfrac{1}{4} d_1(d_1-1)(d_1-2)(d_1-3) \bigl( (d_2-1)(d_2-2)-2g
    \bigr) \\
    + \tfrac{1}{4} d_1(d_1-2)(d_1-3)(d_1+3)d_2(d_2-1) \, ,
  \end{multline*}
  so the number of lines bitangent\index{bitangent} to $S$ and secant to $C$
  is $\gamma$.
\end{example}

\section{Singular Loci of Congruences}
\label{05sec:singLoci}

This section investigates the singular points of the secant, bitangent, and
inflectional congruences.  We begin with the singularities of the secant locus
of a smooth irreducible curve.
 
\begin{proposition}
  \label{05prop:bis3pts}
  Let $C$ be a nondegenerate smooth irreducible curve\index{curve!space} in
  $\PP^3$.  If $L$ is a line that intersects the curve $C$ in three or more
  distinct points, then the line $L$ corresponds to a singular point in
  $\operatorname{Sec}(C)$.
\end{proposition}

\begin{proof}
  The symmetric square $C^{(2)}$ is the quotient of $C \times C$ by the action
  of the symmetric group $\mathfrak{S}_2$, so points in this projective
  variety are unordered pairs of points on $C$; see
  \cite[pp.~126--127]{05harris}.  The map
  $\varpi \colon C^{(2)} \to \operatorname{Sec}(C)$, defined by sending
  $\{x,y\}$ to the line spanned by the points $x$ and $y$ if $x \neq y$ or to
  the tangent line $T_x(C)$ if $x = y$, is a birational morphism.  Since
  $|L \cap C| \geq 3$, the fibre $\varpi^{-1}(L)$ is a finite set containing
  more than one element.  Hence, $\varpi^{-1}(L)$ is not connected and the
  Zariski Connectedness Theorem~\cite[Sect.~III.9.V]{05mumford}
  \index{Zariski~Connectedness~Theorem} proves that $\operatorname{Sec}(C)$ is
  singular at $L$. \qed
\end{proof}

\begin{lemma}
  \label{05lem:skewsym}
  If $f \in \CC[\![z,w]\!]$ satisfies $f(z,w)=-f(w,z)$, then the linear form
  $z-w$ divides the power series $f$.
\end{lemma}

\begin{proof}
  We write the formal power series $f$ as a sum of homogeneous polynomials
  $f = \sum_{i \in \NN} f_i$. Since we have $f(z,w) + f(w,z) = 0$, it follows
  that, in each degree $i$, we have $f_i(z,w) + f_i(w,z) = 0$.  In particular,
  we see that $f_i(w, w) = 0$. If we consider $f_i(w,z)$ as a polynomial in
  the variable $z$ with coefficients in $\CC[w]$, it follows that $w$ is a
  root of $f_i$.  Thus, we conclude that $z-w$ divides $f_i$ for all
  $i \in \NN$.  \qed
\end{proof}

\begin{theorem}
  \label{05thm:bis1pt}
  Let $C$ be a nondegenerate smooth irreducible curve\index{curve!space} in
  $\PP^3$.  If a point in $\operatorname{Sec}(C)$ corresponds to a line $L$
  that intersects $C$ in a single point $x$, then the intersection
  multiplicity of $L$ and $C$ at $x$ is at least $2$. Moreover, the line $L$
  corresponds to a smooth point of $\operatorname{Sec}(C)$ if and only if the
  intersection multiplicity is exactly $2$.
\end{theorem}

We thank Jenia Tevelev for help with the following proof.

\begin{proof}
  Suppose the line $L$ intersects the curve $C$ at the point $x$ with
  multiplicity $2$. Without loss of generality, we may work in the affine open
  subset with $x_3 \neq 0$, and we assume that $x = (0 : 0 : 0 : 1)$ and
  $L = \variety(x_1, x_2)$.  Since $C$ is smooth, there is a local analytic
  isomorphism $\varphi$ from a neighbourhood of the origin in $\A^1$ to a
  neighbourhood of the point $x$ in $C$.  The map $\varphi$ will have the form
  $\varphi(z) = \bigl( \varphi_0(z), \varphi_1(z), \varphi_2(z) \bigr)$ for
  some $\varphi_0, \varphi_1, \varphi_2 \in \CC[\![ z ]\!]$.  We have
  $\varphi_0'(0) \neq 0$ and $\varphi_1'(0) = \varphi_2'(0) = 0$ because $L$
  is the tangent to the curve $C$ at $x$.  After making an analytic change of
  coordinates, we may assume that
  $\varphi(z) = \bigl( z, \varphi_1(z), \varphi_2(z) \bigr)$.  As $L$ is a
  simple tangent, at least one of $\varphi_1$ and $\varphi_2$ must vanish at
  $0$ with order exactly $2$.  Hence, we may assume that
  $\varphi_1(z) = z^2 + z^3 f(z)$ and $\varphi_2(z) = z^2 g(z)$ for some
  $f,g \in \CC[\![ z ]\!]$.  The line spanned by the distinct points
  $\varphi(z)$ and $\varphi(w)$ on the curve $C$ is given by the row space of
  the matrix
  \[
    \begin{bmatrix}
      z & z^2 + z^3 f(z) & z^2 g(z) & 1 \\
      w & w^2 + w^3 f(w) & w^2 g(w) & 1 \\      
    \end{bmatrix} \, .
  \]
  The Pl\"ucker coordinates\index{Plucker@Pl\"ucker!coordinates} are
  skew-symmetric power series, so Lemma~\ref{05lem:skewsym} implies that they
  are divisible by $z-w$.  In particular, if $f(z) = \sum_i a_i z^i$, then we
  have $p_{0,3} = z-w$,
  \begin{align*}
    p_{0,1} 
    &= z \bigl( w^2 + w^3 f(w) \bigr) - w \bigl( z^2 + z^3 f(z) \bigr)
      = -zw(z-w) \Bigl( 1 + \textstyle\sum\limits_i a_i
      \textstyle\sum\limits_{j=0}^{i+1} w^j z^{i+1-j} \Bigr) \, , \\
    p_{1,3}
    &= z^2 + z^3 f(z) - w^2 - w^3 f(w) = (z-w) \Bigl( z + w +
      \textstyle\sum\limits_i a_i \textstyle\sum\limits_{j=0}^{i+2} z^j
      w^{i+2-j} \Bigr) \, .
  \end{align*}
  The symmetric square $(\A^1)^{(2)}$ of the affine line $\A^1$ is a smooth
  surface isomorphic to the affine plane $\A^2$;
  see~\cite[Example~10.23]{05harris}.  Consider the map
  $\varpi \colon (\A^1)^{(2)} \to \operatorname{Sec}(C)$ defined by sending
  the pair $\{ z, w\}$ of points in $\A^1$ to the line spanned by the points
  $\varphi(z)$ and $\varphi(w)$ if $z \neq w$ or to the tangent line of $C$ at
  $\varphi(z)$ if $z = w$.  In other words, the map $\varpi$ sends
  $\{ z, w \}$ to
  $\left( -zw + h_1(z,w) : \tfrac{p_{0,2}}{z-w} : 1 : \tfrac{p_{1,2}}{z-w} : z +
    w + h_2(z,w) : \tfrac{p_{2,3}}{z-w} \right)$ where
  \begin{xalignat*}{3}
    h_1(z,w) &:= -zw \textstyle\sum\limits_i a_i
    \textstyle\sum\limits_{j=0}^{i+1} w^j z^{i-j+1} & & \text{and} & h_2(z,w)
    &:= \textstyle\sum\limits_i a_i \textstyle\sum\limits_{j=0}^{i+2} z^j
    w^{i+2-j} \, .
  \end{xalignat*}
  Since the forms $zw$ and $z+w$ are local coordinates of $(\A^1)^{(2)}$ in a
  neighbourhood of the origin, we conclude that $\varpi$ is a local
  isomorphism and $\operatorname{Sec}(C)$ is smooth at the point corresponding
  to $L$.

  Suppose the line $L$ intersects the curve $C$ at the point $x$ with
  multiplicity at least $3$.  It follows that the line $L$ is contained in the
  Zariski closure of the set of lines that intersect $C$ in at least three
  points or that intersect $C$ in two points, one with multiplicity at least
  $2$.  By Proposition~\ref{05prop:bis3pts} and Lemma~2.3 in \cite{05arrondo},
  we conclude that the line is singular in $\operatorname{Sec}(C)$.  \qed
\end{proof}

\begin{corollary}
  \label{05cor}
  Let $C$ be a nondegenerate smooth irreducible curve\index{curve!space} in
  $\PP^3$. If the line $L$ corresponds to a point in $\operatorname{Sec}(C)$,
  then $L$ corresponds to a singular point of $\operatorname{Sec}(C)$ if and
  only if one of the following three conditions is satisfied: 
  \begin{itemize}
  \item the line $L$ intersects the curve $C$ in $3$ or more distinct points,
  \item the line $L$ intersects the curve $C$ in exactly $2$ points and $L$ is
    the tangent line to one of these two points,
  \item the line $L$ intersects the curve $C$ at a single point with
    multiplicity at least $3$.  
  \end{itemize}
\end{corollary}

\begin{proof}
  Combine Proposition~\ref{05prop:bis3pts}, Lemma~2.3 in \cite{05arrondo}, and
  Theorem~\ref{05thm:bis1pt}. \qed
\end{proof}


Analogously, we want to describe the singularities of the inflectional locus
$\operatorname{Infl}(S)$ and the bitangent locus $\operatorname{Bit}(S)$ of a
surface $S \subset \PP^3$.

\begin{theorem}
  \label{05thm:singInfl}
  If $S \subset \PP^3$ is an irreducible smooth surface of degree at least $4$
  which does not contain any lines, then the singular locus of
  $\operatorname{Infl}(S)$ corresponds to lines which either intersect $S$
  with multiplicity at least $3$ at two or more distinct points, or
  intersect $S$ with multiplicity at least $4$ at some point.
\end{theorem}

\begin{proof}
  We consider the incidence variety
  \[
    \Psi_S := \overline{\{ (x,L) : \text{$L$ intersects $S$ at $x$ with
        multiplicity $3$} \} } \subset S \times \Gr(1, \PP^3) \, .
  \]
  The projection $\pi \colon \Psi_S \rightarrow \operatorname{Infl}(S)$,
  defined by $(x,L) \mapsto L$, is a surjective morphism.  Since $S$ does not
  contain any lines, all fibres of $\pi$ are finite and Lemma~14.8 in
  \cite{05harris} implies that the map $\pi$ is finite. Moreover, the general
  fibre of $\pi$ has cardinality one, so $\pi$ is birational.  To apply
  Lemma~\ref{05prop:key}, we need to examine the singularities of $\Psi_S$ and
  the differential of $\pi$.

  Let $f \in \CC[x_0,x_1,x_2,x_3]$ be the defining equation for $S$ in
  $\PP^3$.  Consider the affine chart in $\PP^3 \times \Gr(1, \PP^3)$ where
  $x_0 \neq 0$ and $p_{0,1} \neq 0$.  We may assume
  $x = (1 : \alpha : \beta : \gamma)$ and the line $L$ is spanned by the
  points $(1 : 0 : a : b)$ and $(0 : 1 : c : d)$.  In this affine chart, $S$
  is defined by $g_0(x_1,x_2,x_3) := f(1,x_1,x_2,x_3)$.  As in the proof of
  Theorem~\ref{05thm:singChow}, we have $x \in L$ if and only if
  $a = \beta - \alpha c$ and $b = \gamma - \alpha d$.  Parametrizing the line
  $L$ by $\ell(s,t) := (s : s \alpha + t : s \beta + tc : s \gamma + t d)$ for
  $(s : t) \in \PP^1$ shows that $L$ intersects $S$ with multiplicity at least
  $m$ at $x$ if and only if $f \bigl( \ell(s, t) \bigr)$ is divisible by
  $t^m$.  This is equivalent to
  \[
    \tfrac{\partial}{\partial t} \bigl[ f\bigl( \ell(s,t) \bigr) \bigr]
    \Big|_{(1,0)} = \tfrac{\partial^2}{\partial t^2} \bigl[ f\bigl(
    \ell(s,t) \bigr) \bigr] \Big|_{(1,0)} = \dotsb =
    \tfrac{\partial^{m-1}}{\partial t^{m-1}} \bigl[ f\bigl( \ell(s,t) \bigr)
    \bigr] \Big|_{(1,0)} = 0 \, .
  \]  
  Setting
  $g_k := \bigl[\frac{\partial}{\partial x_1} + c \frac{\partial}{\partial
    x_2} + d \frac{\partial}{\partial x_3} \bigr]^k g_0$ for $k \geq 1$, the
  incidence variety $\Psi_S$ can be written on the chosen affine chart as
  \[
    \bigl\{ (\alpha, \beta, \gamma, a, b, c, d) :
    \text{$g_k(\alpha, \beta, \gamma) = 0$ for $0 \leq k \leq 2$,
      $a = \beta - \alpha c$, $b = \gamma - \alpha d$} \bigr\} \, .
  \]
  As $\dim \Psi_S = 2$, it is smooth at the point $(x,L)$ if and only if its
  tangent space has dimension $2$ or, equivalently, its Jacobian matrix 
  \[
    \begin{bmatrix}
      \frac{\partial g_0}{\partial x_1}(\alpha, \beta, \gamma) &
      \frac{\partial g_0}{\partial x_2}(\alpha, \beta, \gamma) &
      \frac{\partial g_0}{\partial x_3}(\alpha, \beta, \gamma) & 0 & 0 & 0 & 0 \\
      \frac{\partial g_1}{\partial x_1}(\alpha, \beta, \gamma) &
      \frac{\partial g_1}{\partial x_2}(\alpha, \beta, \gamma) &
      \frac{\partial g_1}{\partial x_3}(\alpha, \beta, \gamma) & 0 & 0 &
      \frac{\partial g_0}{\partial x_2}(\alpha, \beta, \gamma) &
      \frac{\partial g_0}{\partial x_3}(\alpha, \beta, \gamma) \\
      \frac{\partial g_2}{\partial x_1}(\alpha, \beta, \gamma) &
      \frac{\partial g_2}{\partial x_2}(\alpha, \beta, \gamma) &
      \frac{\partial g_2}{\partial x_3}(\alpha, \beta, \gamma) & 0 & 0 &
      2 \frac{\partial g_1}{\partial x_2}(\alpha, \beta, \gamma) & 
      2 \frac{\partial g_1}{\partial x_3}(\alpha, \beta, \gamma) \\
      -c & 1 & 0 & -1 & 0 & -\alpha & 0 \\
      -d & 0 & 1 & 0 & -1 & 0 & -\alpha
    \end{bmatrix}
  \]
  has rank five. Since $S$ is smooth, the first $2$ and the last $2$ rows of
  the Jacobian matrix are linearly independent.  If $\Psi_S$ is singular at
  $(x, L)$, then the third row is a linear combination of the others;
  specifically, there exist scalars $\lambda, \mu \in \CC$ such that
  $\frac{\partial g_2}{\partial x_j}(\alpha, \beta, \gamma) = \lambda
  \frac{\partial g_1}{\partial x_j}(\alpha, \beta, \gamma) + \mu
  \frac{\partial g_0}{\partial x_j}(\alpha, \beta, \gamma)$ for
  $1 \leq j \leq 3$.  It follows that
  $g_3(\alpha, \beta, \gamma) = \lambda g_2(\alpha, \beta, \gamma) + \mu
  g_1(\alpha, \beta, \gamma) = 0$.  Thus, the line $L$ intersects the
  surface $S$ at the point $x$ with multiplicity at least $4$ if
  $\Psi_S$ is singular at $(x, L)$.

  It remains to show that the differential
  $d_{(x,L)} \pi \colon T_{(x,L)}(\Psi_S) \to T_L\bigl(\operatorname{Infl}(S)
  \bigr)$ is not injective if and only if the line $L$ intersects the surface
  $S$ at the point $x$ with multiplicity at least $4$.  The differential
  $d_{(x,L)} \pi$ sends every element in the kernel of the Jacobian matrix to
  its last four coordinates.  This map is not injective if and only if the
  kernel contains an element of the form $\begin{bmatrix} 
    * & * & * & 0 & 0 & 0 & 0
  \end{bmatrix}^{\transpose} \neq 0$.  Such an element belongs to the kernel
  if and only if it equals $\begin{bmatrix}
    \lambda & c \lambda & d \lambda & 0 & 0 & 0 & 0
  \end{bmatrix}^{\transpose}$ for some $\lambda \in \CC \setminus \{ 0 \}$ and
  $g_1(\alpha, \beta, \gamma) = g_2(\alpha, \beta, \gamma) = g_3(\alpha,
  \beta, \gamma) = 0$. This shows that the line $L$ intersects the surface $S$
  at the point $x$ with multiplicity at least $4$ if and only if $d_{(x,L)}$
  is not injective.

  Finally, the fibre $\pi^{-1}(L)$ consists of more than one point if and only
  if $L$ intersects $S$ with multiplicity at least $3$ at two or more distinct
  points, so Lemma~\ref{05prop:key} completes the proof.  \qed
\end{proof}

\subruninhead{Proof of Theorem~\ref{05mainthm}.}  The first part related to
the curve $C$ is an amalgamation of Theorem~\ref{05thm:singChow},
Theorem~\ref{05thm:bidegSecant}, Theorem~\ref{05thm:bis1pt}, and
Corollary~\ref{05cor}.  Similarly, the second part related to the surface $S$
is an amalgamation of Theorem~\ref{05thm:singHurwitz},
Theorem~\ref{05thm:bidegTangent}, and Theorem~\ref{05thm:singInfl}. \qed

\begin{proposition}
  Let $S \subset \PP^3$ be a general irreducible surface of degree at least
  $4$.  If $L$ is a line that is tangent to $S$ at three or more distinct
  points, then the line $L$ corresponds to a singular point of
  $\operatorname{Bit}(S)$.
\end{proposition}

\begin{proof}
  As in the proof of Proposition~\ref{05prop:bis3pts}, the symmetric square
  $S^{(2)}$ is the quotient of $S \times S$ by the action of the symmetric
  group $\mathfrak{S}_2$.  The projection $\varpi$ from
  \[
    \overline{ \big\{ ( \{x,y\}, L) : \text{$x \neq y$, $x, y \in L \subset
        T_x(S) \cap T_y(S)$} \big\}} \subset S^{(2)} \times \Gr(1,\PP^3)
  \]
  onto $\operatorname{Bit}(S)$, defined by sending the pair
  $(\{ x, y \},L) \mapsto L$ is a birational morphism.  The fibre
  $\varpi^{-1}(L)$ is a finite set containing more than one element if $L$ is
  tangent to $S$ in at least three distinct points.  Hence, $\varpi^{-1}(L)$
  is not connected and the Zariski Connectedness
  Theorem~\cite[Sect.~III.9.V]{05mumford}
  \index{Zariski~Connectedness~Theorem} proves that $\operatorname{Bit}(S)$ is
  singular at $L$. \qed
\end{proof}


We do not yet have a full understanding of points in $\operatorname{Bit}(S)$
for which the corresponding lines have an intersection multiplicity greater
than $4$ at a point of $S$.  We know that a line $L$ that is tangent to the
surface $S$ at exactly two points corresponds to a smooth point in
$\operatorname{Bit}(S)$ if and only if the intersection multiplicity of $L$
and $S$ at both points is exactly $2$.  Moreover, given a line $L$ that is
tangent to $S$ at a single point, the intersection multiplicity of $L$ and $S$
at this point is at least $4$, and the line $L$ corresponds to a smooth point
of $\operatorname{Bit}(S)$ when the multiplicity is exactly four; see
\cite[Lemma~4.3]{05arrondo}.  To complete this picture, we make the following
prediction.

\begin{conjecture}
  Let $S \subset \PP^3$ be a general irreducible surface of degree at least
  $4$.  If a point in the bitangent congruence $\operatorname{Bit}(S)$
  \index{bitangent} corresponds to a line $L$ that is tangent to $S$ at a
  single point $x$ such that the intersection multiplicity of $L$ and $S$ at
  $x$ is at least $5$, then $L$ corresponds to a singular point of
  $\operatorname{Bit}(S)$.
\end{conjecture}

\begin{acknowledgement}
  This article was initiated during the Apprenticeship Weeks (22
  August--2~September 2016), led by Bernd Sturmfels, as part of the
  Combinatorial Algebraic Geometry Semester at the Fields Institute.  We thank
  Daniele Agostini, Enrique Arrondo, Peter B\"urgisser, Diane Maclagan, Emilia
  Mezzetti, Ragni Piene, Jenia Tevelev, and the anonymous referees for helpful
  discussions, suggestions and hints.  Kathl\'en Kohn was supported by a
  Fellowship from the Einstein Foundation Berlin, Bernt Ivar Utst\o{}l
  N\o{}dland was supported by NRC project 144013, and Paolo Tripoli was
  supported by EPSRC grant EP/L505110/1.
\end{acknowledgement}

\label{05:end}

\end{document}